\documentclass[12pt,twoside,reqno]{amsart}
\usepackage{amssymb,amsmath,amstext,amsthm,amsfonts}
\usepackage[ansinew]{inputenc} 
\usepackage{graphicx}
\usepackage[mathscr]{eucal}
\usepackage{hyperref}

\newcommand{\blob}{\rule[.2ex]{.8ex}{.8ex}}

\newcommand{\R}{\mathbb{R}}
\newcommand{\C}{\mathbb{C}}
\newcommand{\N}{\mathbb{N}}
\newcommand{\Z}{\mathbb{Z}}
\newcommand{\Q}{\mathbb{Q}}
\newcommand{\T}{\mathbb{T}}
\newcommand{\SL}{{\rm SL}}
\newcommand{\GL}{{\rm GL}}
\newcommand{\symp}{{\rm sp}}
\newcommand{\Sym}{{\rm Sym}}
\newcommand{\Mat}{{\rm Mat}}
\newcommand{\zero}{{\rm O}}
\newcommand{\tr}{\mbox{tr}}

\newcommand{\Jet}{\mathscr{J}{\rm et}}
\newcommand{\V}{\mathscr{V}}

\newcommand{\ep}{\epsilon}  
\newcommand{\epob}{\overline{\ep_0}} 
\newcommand{\epub}{\overline{\ep_1}} 
\newcommand{\la}{\lambda}

\newcommand{\om}{\omega}
\newcommand{\Wf}{W^\flat}
\newcommand{\Ws}{W^\sharp}
\newcommand{\vpsi}{\vec{\psi}}
\newcommand{\Arho}{C^{\om}_{\rho}}

\newcommand{\LE}[1]{L^{({#1})}}  

\newcommand{\uk}{u^{(k)}}
\newcommand{\Sk}{S^{(k)} }
\newcommand{\gak}{\gamma^{(k)}}
\newcommand{\Dmat}{U}
\newcommand{\adjW}{\tilde{W}}

\newcommand{\Ae}{A_{\la, E}}

\newcommand{\abs}[1]{\left\vert{#1}\right\vert} 
\newcommand{\norm}[1]{\left\|{#1}\right\|} 

\newcommand{\Ann}{\mathscr{A}}
\newcommand{\strip}{\mathscr{A}}

\newcommand{\Circ}{\mathscr{C}}
\newcommand{\Herm}{{\rm Herm}}

\newcommand\restr[2]{{
  \left.\kern-\nulldelimiterspace 
  #1 
  \vphantom{\big|} 
  \right|_{#2} 
  }}

\theoremstyle{plain}
\newtheorem{theorem}{Theorem}[section]
\newtheorem{proposition}{Proposition}[section]
\newtheorem{corollary}[proposition]{Corollary}
\newtheorem{lemma}[proposition]{Lemma}
\newtheorem{definition}{Definition}[section]
\newtheorem{remark}{Remark}[section]

\numberwithin{equation}{section}

\title[Higher dimensional quasiperiodic cocycles]{ Positive Lyapunov exponents for higher dimensional quasiperiodic cocycles }

\date{}

\begin{document}

\author[P. Duarte]{Pedro Duarte}
\address{departamento de matem\'atica and cmaf \\
faculdade de ci\^encias\\
universidade de lisboa\\
lisboa, portugal 
}
\email{pduarte@ptmat.fc.ul.pt}

\author[S. Klein]{Silvius Klein}
\address{cmaf,
faculdade de ci\^encias\\
universidade de lisboa\\
lisboa, portugal 
\& imar, 
bucharest, romania }
\email{silviusaklein@gmail.com}

\begin{abstract}We consider an $m$-dimensional analytic cocycle \\ $\displaystyle \T \times \R^m \ni (x, \vec{\psi}) \mapsto (x + \om, A (x) \cdot \vpsi) \in \T \times \R^m$, where $\om \notin \Q$ and $A \in C^\om (\T, \Mat_m (\R))$. Assuming that the  $d \times d$ upper left corner block of $A$ is typically large enough, we prove that the $d$ largest Lyapunov exponents associated with this cocycle are bounded away from zero. The result is uniform relative to certain measurements on the matrix blocks forming the cocycle. As an application of this result we obtain nonperturbative (in the spirit of Sorets-Spencer theorem) positive lower bounds of the nonnegative Lyapunov exponents for various models of band lattice Schr\"{o}dinger operators.
\end{abstract}

\maketitle

\section{Introduction, definitions and notations}\label{introduction}

In this paper we consider a higher dimensional analytic linear cocycle $A_\la (x)$ or a family  $A_{\la, E} (x)$ indexed by the parameter $E \in \R$  of such cocycles. We make assumptions on a designated upper left corner block of this (family of) cocycle(s) which ensure that this block is ``typically large'' enough. Under these assumptions we prove that the $d$ largest Lyapunov exponents of this (family of) cocycle(s) are positive, where $d$ is the dimension of the designated large block. 
The result is nonperturbative (i.e. independent of the underlying frequency) and uniform relative to certain measurements on the matrix blocks forming the cocycle. 

By a ``typically large'' matrix block we understand one which consists of a \emph{transversal} matrix-valued function or which factors out as a product of transversal matrix-valued functions,  multiplied by a large enough constant $\la$.
The \emph{transversality} condition on such a  matrix-valued function will either mean that it is not identically singular, or else that it has no
constant eigenvalues. We establish the genericity of this last condition both in a topological sense (meaning
for an open and dense set of potentials) and in some strong algebraic geometric sense (where it refers to the complement of some `algebraic subvariety' of $\infty$-codimension in the  space of analytic potentials). 


A natural problem related to the results in this paper is finding a suitable generic condition on the cocycle that would ensure not only positivity of the Lyapunov exponents, but also gaps between the Lyapunov exponents. This is the assumption under which H\"{o}lder continuity of the Lyapunov exponents has been recently established in \cite{Schlag}.


Choosing the matrix blocks forming the family of cocycles $A_{\la, E} (x)$ appropriately, we show that our result on positivity of the Lyapunov exponents applies to cocycles associated with various models of Schr\"{o}dinger operators on a band lattice (which some authors call a strip).  These models include all finite range hopping Schr\"odinger operators, both on the integer lattice and on band lattices. The large enough constant $\la$ will be the coupling constant, while the parameter $E$ will be the energy corresponding to these models of Schr\"odinger operators.

In the (one-dimensional) integer lattice case, the Lyapunov exponent characterizes the absolutely continuous spectrum of the Schr\"odinger operator completely. This is due to Kotani's theory, which implies the absence of absolutely continuous spectrum when the (largest) Lyapunov exponent is positive. Kotani's theory has an extension to the standard band lattice model (see \cite{Simon-Kotani}), so at least in that case, our result implies absence of absolutely continuous spectrum. While the Lyapunov exponents cannot distinguish between point and singular spectra, we expect that under suitable assumptions, pure point spectrum with exponentially decaying eigenfunctions (i.e. Anderson localization) will hold for the general band lattice models we consider here.

Standard results on the positivity of the Lyapunov exponent for one-dimensional Schr\"{o}dinger operators (i.e. $\SL_2(\R)$ Schr\"{o}dinger cocycles) are due to M. Herman (see \cite{H}) and E. Sorets, T. Spencer (see \cite{S-S}). 
Both results use complexification and subharmonicity in order to avoid the set $\{ x \colon v (x) - E \approx 0 \}$ (see also chapter 3 in \cite{B}). 

Quasi-periodic band lattice Schr\"{o}dinger operators  lead to higher dimensional Schr\"{o}dinger cocycles.  I. Ya. Goldsheid and E. Sorets (see \cite{SG}) proved positivity of the Lyapunov exponents for such a model where the potential has typically large enough quasi-periodic diagonal and constant off-diagonal entries. J. Bourgain and S. Jitomirskaya (see \cite{BJ}) proved Anderson localization for such a model where the potential function is a quasi-periodic diagonal matrix with typically large enough entries.

Our paper extends both in scope and in approach Sorets-Spencer's theorem to higher dimensional general cocycles. In particular, it also applies to complex valued cocycles (see section~\ref{consequences}).

\smallskip

We now introduce the basic definitions and notations used in the paper. 

The letter $\T$ will refer to the additive group $\T=\R/\Z$.
Given an irrational number $\omega\in\R\setminus\Q$,
consider the translation $T=T_\omega:\T\to\T$, $T x = x+\omega\, {\rm mod} \, \Z$,
which is an ergodic transformation with respect to the Haar measure $dx$ on $\T$.
Any measurable function $A:\T\to \Mat_m(\R)$ determines a skew-product map
$F:\T\times\R^m\to \T\times \R^m$ defined by $F(x,v)=(T x, A(x)\,v)$.
The dynamical  system underlying such a map is called a {\em linear cocycle}
over the translation $T$.
Sometimes, when $T$ is fixed, the measurable function $A$ is also referred to as a linear cocycle.
The iterates of $F$ are given by  $F^n(x,v)=(x+n\,\omega, M_n(x)\,v)$,  where
$$M_n(x)= A(x+(n-1)\omega) \, \ldots \, A(x+\omega)\,A(x)$$
The cocycle is called {\em integrable }\, when
$\int_\T \log^+ \norm{A(x)}\, d x <\infty$ and
$\int_\T \log^+ \norm{A(x)^{-1}}\, d x <\infty$,
where \, $\log^+(x) = \max\{\log x, 0\}$.
In 1965 Oseledets proved his famous Multiplicative Ergodic Theorem, 
which when applied to the previous class of cocycles  says that if 
$A$ is integrable  then there are: numbers $L^{(1)}\geq L^{(2)} \geq \ldots \geq L^{(m)}$,
 an $F$-invariant measurable decomposition
$\R^m=\oplus_{j=1}^{\ell} E^{(j)}_x$,
and a non decreasing surjective map $k:\{1,\ldots, m\}\to \{1,\ldots, \ell\}$ such that
for almost every $x\in \T$, every $1\leq i\leq m$ and every $v\in E^{(k_i)}_x$,
$$\LE{i} = \lim_{n\to\pm \infty}\frac{1}{n}\,\log \norm{ M_n(x)\, v}$$ 
Moreover,  $L^{(i)}=L^{(i+1)}$ if and only if $k_{i}=k_{i+1}$,
and the subspace  $E_x^{(j)}$ has dimension equal to $\#k^{-1}(j)$.
The numbers $\LE{i}$ are called the {\em Lyapunov exponents}  of $F$, or of $A$.
If furthermore $A$ takes values in $\SL_m(\R)$, 
the Lyapunov exponents  satisfy the relation
$\sum_{i=1}^{m} \LE{i}  = 0$, while if $m = 2 d$ and  $A$ takes values in $\symp_d(\R)$, then $\LE{2 d + 1 - i} = - \LE{i}$ for all $1\le i \le d$.

Defining $s_n^{(1)}(x)\geq s_n^{(2)}(x) \geq \ldots \geq s_n^{(m)}(x)$ to be the
singular values of the product matrix $M_n(x)$, it is easy to verify that  the Lyapunov exponents are
\begin{equation}\label{formula-lyap-i}
 \LE{i} = \lim_{n\to\infty}\frac{1}{n}\,\log s^{(i)}_n(x)\quad \text{ for  a.e.}\, x\in \T
\quad (1\leq i\leq m)
\end{equation}
From  Birkhoff's Ergodic Theorem, 
the largest Lyapunov exponent is
\begin{align}
\LE{1}  &= \lim_{n\to\infty}\frac{1}{n}\,\log \norm{ M_n(x) }\;\;
\text{ for  a.e.}\, x\in \T  \nonumber \\
&= \lim_{n\to\infty}\frac{1}{n}\,\int_\T \log \norm{ M_n(x) }\,dx \label{formula-lyap}
\end{align}


Throughout this paper, we will denote by $\Sym_d(\R)$,  $\Sym_d(\C)$ and $\Herm_d(\C)$ the space of $d\times d$ real symmetric, complex symmetric and respectively Hermitian matrices. If $\mathscr{M}$ is any real vector space, say $\Sym_d(\R)$ or $\Mat_m(\R)$, we denote by 
$C^\omega(\T,\mathscr{M})$  the space of
real analytic functions $V : \T \to \mathscr{M}$.
For any $0<\rho<1$,  
$$\Ann_\rho:=\{\, z\in\C\,:\, 1-\rho \leq \abs{z} \leq 1+\rho \,\}$$
denotes  the $2\rho$-width {\em annulus}. 
Let  $C^\omega_\rho(\T,\mathscr{M})$ (mind the subscript $\rho$)
be the subspace of all functions $V(x) \in C^\omega(\T,\mathscr{M})$ which have a  holomorphic extension   
$V(z):{\rm int}(\Ann_\rho)\to\mathscr{M}^\C=\mathscr{M}\otimes \C$,
which are continuous up  to the boundary of $\Ann_\rho$.
Endowed with the norm
$\norm{V}_\rho =\max_{z\in \Ann_\rho} \norm{V(z)}$,
the space $C^\omega_\rho(\T,\mathscr{M})$ becomes a Banach space.
Given any submanifold $\mathscr{N}\subset \mathscr{M}$, let 
$C^\omega_\rho(\T,\mathscr{N})$ be the Banach submanifold of all functions
$V\in C^\omega_\rho(\T,\mathscr{M})$ such that
$V(x)\in\mathscr{N}$, for every $x\in\T$. 
Elements in the spaces $C^\omega_\rho(\T,\SL_m(\R))$ 
and $C^\omega_\rho(\T,\SL_m(\C))$  are called
{\em analytic cocycles}. The blocks forming an analytic cocycle are matrix-valued analytic functions. 

We introduce some measurements on a matrix-valued analytic function $V:\T\to\Mat_d(\R)$.
We denote by $N_\rho(V)$ the number of zeros of
$\det(V(z))$ in $\Ann_\rho$, and by
$\beta_\rho(V)$ the minimum value on  $\Ann_\rho$ of the holomorphic function obtained by factoring out all zeros of
$\det(V(z))$. Moreover, for  $V:\T\to\Sym_d(\R)$ we consider  the functions
$\widehat{N}_\rho(V)=\max_{E\in\R} N_\rho(V-E\cdot I)$ and 
$\widehat{\beta}_\rho(V)=\min_{E\in\R}  \beta_\rho(V-E\cdot I)$.
For every 
$V\in C^\omega_\rho(\T,\Sym_d(\R))$ which satisfies a generic transversality condition, we will show that $\widehat{N}_\rho(V)<+\infty$ and
$\widehat{\beta}_\rho(V)>0$  and moreover,
that these quantities depend continuously on $V$ (see section~\ref{bounds} for details).


The paper is organized as follows. In section~\ref{main_section} we present the two main statements, Theorem~\ref{main} and Theorem~\ref{main-allE}, and the main application Theorem~\ref{main-ap}. The following two sections describe the assumptions made in the main statements on the designated upper left corner block: in section~\ref{genericity_section} we show that these assumptions are generic in a strong sense, while in section~\ref{bounds} we define certain measurements on the function(s) forming this block and we show that they depend continuously on it. The subsequent two sections contain the main technical tools used in the proof, described in general terms: in section~\ref{growth-lemma} we prove a growth result for products of block matrices that have a designated ``large'' block, while section~\ref{subharmonic_section} contains an estimate on the mean of a subharmonic function. Section~\ref{proof} contains the proof of the main statements, while in section~\ref{consequences} we show that our statements apply to cocycles associated to general band lattice Schr\"{o}dinger operators with both real and complex entries.

\medskip

\section{The main statements}\label{main_section}
We consider two families of $m$-dimensional  cocycles of the form  
\begin{equation}\label{(T,A)}
(T, A) \colon \T \times \R^m \ni (x, \vec{\psi}) \mapsto (x + \om, A (x) \cdot \vpsi) \in \T \times \R^m
\end{equation}
where $\om \in \R \setminus \Q$  and $A \in \Arho (\T, \Mat_m (\R))$.

In the first family of cocycles,  the matrix $A$ depends on a  coupling 
constant $\la$ and consists of block matrices of the form 
\begin{equation}\label{cocycle}
A_\lambda (x) = \left[ \begin{array}{ccc}
\lambda\, V  ( x )     & &   W^\flat (x)  \\ 
& & \\ 
W^\sharp (x)  & &  O (x) \\  \end{array} \right]  
\end{equation}
with the upper left corner being a square $d$-dimensional matrix block ($1 \le d < m$).

In the second  family of cocycles, the matrix $A$ depends on the coupling constant 
 $\la$ and on an energy parameter $E$,  and consists of block matrices of the form 
\begin{equation}\label{cocycle-allE}
A_{\la,E} (x) = \left[ \begin{array}{ccc}
\la \, \Dmat(x) (V  ( x ) - E \cdot I)  & &   W^\flat (x)  \\ 
& & \\ 
W^\sharp (x)  & &  O (x) \\  \end{array} \right]  
\end{equation}
where the upper left corner is a $d$-dimensional matrix block ($1 \le d < m$) with $V$ symmetric.

For both families we prove that if the  matrix-valued functions $V (x)$,
respectively $V(x)$ and $\Dmat(x)$, are transversal and if the  coupling constant $\la$ is large enough, then the $d$ largest Lyapunov exponents associated to these cocycles are bounded away from zero. 

The result is nonperturbative, in the sense that the threshold $\la_0$ on the size of the coupling constant $\la$ does not depend on the frequency $\om$ but only on certain measurements on the matrix-valued analytic  functions $V (x)$ and $\Dmat(x)$,
and on the sup norms of the other blocks $\Wf (x)$, $\Ws (x)$ and $O(x)$. In particular, the threshold $\la_0$ and the lower bounds on the Lyapunov exponents  are uniform in these measurements, which is what makes the statements below appear more technical.

\medskip

We make the following assumptions on the block matrices that form the first family of cocycles $A_\la (x)$.

\smallskip

$\blob$ Uniform bounds: 
for some constants $N \in \N$, $\beta > 0$ and $B>0$,
\begin{equation}\label{unifbound}
N_\rho (V) \leq  N \ \text{ and } \  \beta_\rho (V) > \beta
\end{equation}
\begin{equation}\label{supnorm}
\max \ \{\norm{ V }_\rho, \norm{ \Wf }_\rho, \norm{ \Ws }_\rho, \norm{ O }_\rho \} \leq B
\end{equation}

\smallskip

$\blob$ Transversality condition:
\begin{equation}\label{TC}
\det ( V (x) )  \not \equiv 0
\end{equation}

\medskip

For the second family of cocycles $\Ae (x)$,  we assume
that $V (x)$ is symmetric, and we make similar but stronger  (i.e. uniform in the parameter $E \in \R$) assumptions on its matrix blocks.
\smallskip

$\blob$ Uniform bounds:
for some constants $N_1, N_2 \in \N$, $\beta_1, \beta_2 > 0$ and $B>0$,
\begin{align}
 N_\rho (\Dmat)  \leq N_1 \ & \text{ and } \  \beta_\rho (\Dmat) > \beta_1 \label{unifbound-allE-U} \\
\widehat{N}_\rho (V) \le N_2 \ & \text{ and } \  \widehat{\beta}_\rho (V) > \beta_2 \label{unifbound-allE-V}
\end{align}
\begin{equation}\label{supnorm-all}
 \max \ \{\norm{ V }_\rho, \norm{ \Dmat }_\rho, \norm{ \Wf }_\rho, \norm{ \Ws }_\rho, \norm{ O }_\rho \}\leq B
\end{equation}

\smallskip

$\blob$ Transversality condition:
\begin{align}
\det (\Dmat (x) )  \not \equiv 0 & \label{TC-U}\\
\det  ( V (x) - E \cdot I )  \not \equiv 0 &\  \text{ for any } \ E \in \R  \label{TC-allE}
\end{align}
The second condition above says that
$V (x)$ has no constant eigenvalues, as functions of $x$.

\medskip

\begin{theorem}\label{main}
Consider the cocycle \eqref{(T,A)} where $A_\la \in \Arho (\T, \Mat_m (\R))$ is defined as in \eqref{cocycle}. We assume the uniform bounds \eqref{unifbound} and \eqref{supnorm} and the transversality condition \eqref{TC}.  

There are constants $\la_0 (\rho, B, N, \beta, d)>0$ 
and $c=c (\rho, B, N, \beta, d)>0$
such that  for  $\la > \la_0,$  the $d$ largest Lyapunov exponents associated with this cocycle are positive:
\begin{equation}\label{poslyap}
L^{(k)}(A_\la) \ge   \log \la - k\,c \quad \text{ for all } \ 1 \le k \le d
\end{equation}
\end{theorem}

\medskip

When the parameter $E$ is fixed, the second family of cocycles $\Ae (x)$ is of the same type as the first. The following theorem says that if we assume instead the stronger transversality condition \eqref{TC-allE}  and the stronger uniform bounds \eqref{unifbound-allE-V} (which have the effect of increasing the size of the constants $\la_0$ and $c$), then the same result holds for the family of cocycles $\Ae (x)$ {\em uniformly} in the  parameter $E$.

\medskip
  
\begin{theorem}\label{main-allE}
Consider the cocycle \eqref{(T,A)} where $A_{\la,E} \in \Arho (\T, \Mat_m (\R))$ is defined as in \eqref{cocycle-allE},
with $V(x)$ symmetric. We assume the
uniform bounds \eqref{unifbound-allE-U}, \eqref{unifbound-allE-V} and \eqref{supnorm-all} 
and the transversality conditions \eqref{TC-U} and  \eqref{TC-allE}.  

Then there are constants $\la_0 (\rho, B, N_1, N_2, \beta_1, \beta_2, d)>0$ 
and $\hat{c} =\hat{c} (\rho, B, N_1, N_2, \beta_1, \beta_2, d)>0$ 
such that  for  $\la >  \la_0,$   the $d$ largest Lyapunov exponents associated with the cocycles $A_{\la,E}$ are positive:
\begin{equation}\label{poslyap-allE}
L^{(k)} (A_{\la,E}) \ge  \log \la -k\, \hat{c}\quad \text{ for all } \ E \in \R \ \text{ and } \ 1 \le k \le d
\end{equation}
\end{theorem}

\smallskip

\begin{remark} \rm{The transversality conditions \eqref{TC}, \eqref{TC-U}, \eqref{TC-allE} are generic in a strong sense (see section \ref{genericity_section}).}
\end{remark}

\begin{remark} \rm{ The symmetry of $V(x)$
in Theorem~\ref{main-allE} is enough for our purposes
but not essential. The same statement holds for a non symmetric matrix-valued function $V(x)$ if we redefine
the measurements $\widehat{N}_\rho(V)$ and
$\widehat{\beta}_\rho (V)$ respectively as a maximum and
a minimum of the same quantities but over complex parameters $E$ (see definition~\ref{hat:functions}). }
\end{remark} 

\smallskip

We now describe the main application of Theorem~\ref{main-allE} to a mathematical physics model.

Let $W(x), R (x), D(x) \in \Mat_d (\R)$ for all $x \in \T$. Assume that $R(x)$ and $D (x)$ are symmetric and denote by $W^T (x)$ the transpose of the matrix $W (x)$. Moreover, for all $n \in \N$, denote 
\begin{equation}\label{wnrndn}
W_n (x) := W (x+n \om), R_n (x) := R (x+n \om), D_n (x) := D (x+n \om)
\end{equation}

Consider the {\em quasi-periodic} Schr\"{o}dinger (or Jacobi, as referred to by other authors) operator $H = H_{\la, x} $ acting on $l^2 (\Z, \R^d)$ by
\begin{equation}\label{J-op}
[ H_{\la, x} \, \vpsi ]_n := - (W_{n+1} (x) \, \vpsi_{n+1} + W^{T}_{n} (x) \, \vpsi_{n-1} + R_{n} (x) \, \vpsi_{n}) + \la \, D_n (x) \, \vpsi_n 
\end{equation}
where $\vpsi = \{ \vpsi_n \}_{n\in\Z} \in l^2 (\Z, \R^d)$ is any state, $x \in \T$ is a parameter that introduces some randomness into the system and $\la > 0$ is a coupling constant.

This model contains all quasi-periodic, finite range hopping Schr\"{o}dinger operators on integer or band integer lattices. The hopping term is given by the ``weighted'' Laplacian: 
\begin{equation}\label{w-Laplace}
[S_x \, \vpsi]_n :=   - W_{n+1} (x) \, \vpsi_{n+1} + W^{T}_{n} (x) \, \vpsi_{n-1} + R_{n} (x) \, \vpsi_{n}
\end{equation}
where the hopping amplitude is encoded by the quasi-periodic matrix valued functions $W_n (x)$ and $R_n (x)$, while the potential is given by the quasi-periodic matrix valued function $\la \, D_n (x)$. The physically more relevant situation is when $D_n (x)$ (hence $D(x)$) is a {\em diagonal} matrix, but our result applies to any symmetric matrices.

The associated Schr\"{o}dinger equation 
$$
H_{\la, x} \, \vpsi = E \, \vpsi
$$
for a (generalized) state  $\vpsi = \{ \vpsi_n \}_{n\in\Z} \subset \R^d$ and energy $E \in \R$, gives rise to a Schr\"{o}dinger cocycle $A_{\la, E} (x)$. Let $\LE{k} (A_{\la, E})$ be the $k$th Lyapunov exponent of this Schr\"{o}dinger cocycle.

\begin{theorem}\label{main-ap}
Consider the Schr\"{o}dinger equation associated to the operator \eqref{J-op}:
\begin{equation}\label{J-eq}
- (W_{n+1} (x) \, \vpsi_{n+1} + W^{T}_{n} (x) \, \vpsi_{n-1} + R_{n} (x) \, \vpsi_{n}) + \la \, D_n (x) \, \vpsi_n 
= E \, \vpsi_n
\end{equation}
where $\vpsi = \{ \vpsi_n \}_{n\in\Z} \subset \R^d$, $E \in \R$, and the hopping amplitude and the potential are defined as in \eqref{wnrndn}.

Assume that $W \in  \Arho (\T, \Mat_d (\R))$, $R \in  \Arho (\T, \Sym_d (\R))$ and $D \in  \Arho (\T, \Sym_d (\R))$. Assume moreover that
\begin{align}
\det [ W (x) ]  \not \equiv 0  \label{TC-W}\\
D (x)   \text{ has no constant eigenvalues}  \label{TC-D}
\end{align}

Then there are constants $\la_0 = \la_0 (W, R, D)$ and $c = c (W, R, D)$ such that if $\la > \la_0$, then the $d$ largest Lyapunov exponents associated with the equation~\eqref{J-eq} have the lower bounds:
\begin{equation}\label{main-ap-c}
\LE{k} (A_{\la, E}) \ge \log \la - c \quad \text{for all } E\in \R, \ 1 \le k \le d
\end{equation}

Moreover, the other $d$ Lyapunov exponents are the additive inverses of the $d$ largest Lyapunov exponents. 
 \end{theorem}
 
 \begin{remark}
 \rm{ Assumptions \eqref{TC-W} and \eqref{TC-D} on the data are generic in a strong sense (see section~\ref{genericity_section}). In physics applications, the entries of the amplitude matrix-valued function $W (x)$ are usually trigonometric polynomials, hence the determinant will have some zeros.
 
 }
 \end{remark}

\medskip

\section{Genericity of the  potential function}\label{genericity_section}

Throughout this section we shall write
$\Sym_d$ and $\Mat_d$ as a short notation for $\Sym_d(\R)$ and $\Mat_d(\R)$, respectively.

Given a Banach space $E$, let us call a {\em finite codimension algebraic subvariety}
any subset $\Sigma\subset E$ such that for some continuous linear
epimorphism $\pi:E\to\R^n$, and for some algebraic subvariety $S\subset \R^n$ we have
$\Sigma=\pi^{-1}(S)$. The {\em codimension} of $\Sigma$ in $E$ is defined to be
the {\em codimension} of $S$ in $\R^n$.
The complement of an algebraic subvariety  $\Sigma\subset E$ with codimension $\geq 1$
is always a {\em prevalent set}, a concept introduced by J. Yorke et al. in~\cite{HSY}.
By definition, $E\setminus \Sigma$ is prevalent, since there is a measure $\mu$ compactly supported on $E$
{\em transverse} to $\Sigma$.  This measure $\mu$ can be taken to be the Lebesgue measure on some manifold $M\subset E$ transversal to $\Sigma$, with $\dim(M)={\rm codim}(\Sigma,E)$.

We say that a potential $V\in C^\omega_\rho(\T,\Mat_d)$ 
  {\em has no constant eigenvalues}\, if there is no
common eigenvalue  $E\in\C$  to all matrices $V(x)$ with $x\in\T$.
The main purpose of this section is to prove the following:

\begin{theorem}\label{generic}
Consider $E=\Mat_d$  or $E=\Sym_d$
and let $\mathscr{V}\subseteq  C^\omega_\rho(\T,E)$
 denote the subset of analytic potentials with no
constant eigenvalues. Then:
\begin{enumerate}
\item[(a)]  $\mathscr{V}$ is open and dense;
\item[(b)]  the complement of $\mathscr{V}$ in $C^\omega_\rho(\T,E)$
is contained in  algebraic subvarieties of arbitrary large codimension  
in $C^\omega_\rho(\T,E)$.
\end{enumerate}
\end{theorem}

\smallskip

Fix $d\geq 1$, and for any $1\leq k\leq d$ consider the $k$-th elementary symmetric function
$$e^d_k(\lambda_1,\ldots, \lambda_d)=(-1)^k\sum_{1\leq i_1<i_2<\ldots < i_k\leq d}
\lambda_{i_1} \lambda_{i_2}\ldots \lambda_{i_k}\;.$$  
We have 
$$e^d_1 (\lambda_1,\ldots, \lambda_d)=-(\lambda_1+ \ldots + \lambda_d)\; \text{ and }\;
e^d_d (\lambda_1,\ldots, \lambda_d)=(-1)^d\lambda_1 \ldots  \lambda_d\;.$$
Define $e_k:\Mat_{d} \to \R$ by $e_k(A)=e^d_k(\lambda_1,\ldots, \lambda_d)$,
where $\lambda_1,\ldots, \lambda_d$ are the eigenvalues of $A$.
Note that $e_1(A)=-\tr(A)$ and $e_d(A)=\pm\det(A)$.
By the considerations at the end of section~\ref{growth-lemma},
we have  $e_k(A)=(-1)^k \tr(\wedge_k A)$.
Hence, for each $1\leq k\leq d$, $e_k(A)$ is a homogeneous polynomial of degree $k$ in the
entries of $A$, and for every $E\in\R$ and $A\in \Mat_{d}$,
\begin{equation}\label{charac:eq}
 \det(A-E\,I)=\sum_{k=0}^d  e_k(A)\, E^{d-k} \;.
\end{equation}
Set $\lambda_E: \C^d\to\C$ to be the affine form 
$$\lambda_E(x_1,\ldots, x_d)=E^d + \sum_{k=1}^d x_k \, E^{d-k}\;,$$
and $e_\ast:\Mat_d\to \R^d$, the non-linear map
 $e_\ast(A)=(e_1(A),\ldots, e_d(A))$. 

For $E=\Mat_d$ and $E=\Sym_d$, define $\Jet^d(E)=E^{d+1}$ as the space of $d$-jets of  $E$-valued one variable functions. The $d$-jet of a function $V:\T\to E$, at a point $x\in\T$,
is the vector ${\rm jet}^d_x(V)=(V(x),V'(x),\ldots, V^{(d)}(x))$.
Each potential $V\in C^\omega_\rho(\T,E)$  induces an analytic curve
${\rm jet}^d(V):\T\to \Jet^d(E)$ in the space of $d$-jets.

Next we define a map $\mathscr{I}:\Jet^d(E)\to \Mat_{d}$ setting
$$ \mathscr{I}(A_0,A_1,\ldots, A_d)= \left(  \frac{d^i}{d t^i}\left[ e_j\left(
  \sum_{k=0}^d \frac{t^ k}{k!}\,A_k\right) \right]_{t=0}\right)_{1\leq i,j\leq d}$$
The map $\mathscr{I}$ is defined so that
  $\mathscr{I}({\rm jet}^d_x(V))$ is the matrix with rows  
  $(e_\ast\circ V)'$, $(e_\ast\circ V)''$, $\ldots$, $(e_\ast\circ V)^{(d)}$.
We say that a potential $V:\T\to E$ is {\em non-degenerate} at a point $x\in\T$ if
$\det\mathscr{I}({\rm jet}^d_x(V))\neq 0$, otherwise  $V$ is said to be {\em degenerate} at $x$.

\begin{proposition}\label{non-degenerate:nctev}
If a potential $V$ is non-degenerate at some point $x\in\T$ then
$V$ has no constant eigenvalues.
\end{proposition}

\begin{proof}
Assume $E\in\C$ is a common eigenvalue to all matrices $V(x)$.
Then $\lambda_E(e_\ast (V(x))=\det(V(x)-E\,I)\equiv 0$.
This implies that the range of $e_\ast \circ V$ is contained in the hyperplane
$\{\lambda_E=0\}$. Hence the rows  of $\mathscr{I}({\rm jet}^d_x(V))$,
$(e_\ast\circ V)'(x),\ldots, (e_\ast\circ V)^{(d)}(x)$,
are linearly dependent. Thus $V$ is degenerate at every point  $x\in\T$.
\end{proof}

\begin{proof}[Proof of Theorem~ \ref{generic}]
Let us first prove part (b).
Fix $N$ distinct points $x_1, x_2, \ldots, x_N\in\T$ and consider the  linear map
$J:C^\omega_\rho(\T,E)\to \Jet^d(E)^N$  defined by
$J(V)=\{\,{\rm jet}^d_{x_i}(V)\,\}_{1\leq i\leq N}$. Clearly $J$ is a continuous epimorphism.
Note that given a finite set of $d$-jets at the points 
$x_1, x_2, \ldots, x_N$,
we can always interpolate them with a trigonometric polynomial
with coefficients in the space $E$.
Define now $ \Sigma$ to be the set of all families of $d$-jets
$ \{\underline{A}_i\}_{1\leq i\leq N}$ such that $\det(\mathscr{I}( \underline{A}_i ))=0$,
for every $1\leq i\leq N$. The set $\Sigma$ is an algebraic variety of codimension $N$
in $\Jet^d(E)^N$.
Therefore, $\mathcal{S}= J^{-1}(\Sigma)$ is a algebraic subvariety of  codimension $N$ in
$C^\omega_\rho(\T,E)$. By proposition~\ref{non-degenerate:nctev},
every potential $V\in C^\omega_\rho(\T,E)$ with a constant eigenvalue
must be contained in $\mathcal{S}$.
Finally, the density part in (a) is a direct consequence of (b). To show that
$\mathscr{V}$ is open, note that the $i$th
eigenvalue $\lambda_i(x)$ of an analytic potential $V(x)$ is a  continuous function of $x$ and that
if $\widetilde{\lambda}_i(x)$ stands for the $i$th eigenvalue of another potential $\widetilde{V}(x)$, then we have $\vert \lambda_i(x)-\widetilde{\lambda}_i(x)\vert \leq \|V(x)-\widetilde{V}(x)\|$ (see for instance Lemma B.4 in~\cite{Chan}).
Hence if $\lambda_i(x)$ is not constant and if $\widetilde{V}$ is close
enough to $V$ then $\widetilde{\lambda}_i(x)$ cannot be constant either.
\end{proof}

\begin{corollary} 
The set of potentials $\Dmat\in C^\omega_\rho(\T,\Mat_d)$ such that \\ 
$\det(\Dmat(x)) \not \equiv 0$
satisfies  conditions (a) and (b) of Theorem~\ref{generic}.
\end{corollary}

\begin{proof}
If $\det(\Dmat(x))\equiv 0$ then $0$ is a constant eigenvalue of $\Dmat(x)$.
\end{proof}

\medskip

\section{Uniform bounds on analytic functions}\label{bounds}
Consider the compact   {\em annulus} of width $2\,r > 0$ 
$$ \Ann_r=\{\, z\in\C\,:\, 1-r \leq \abs{z}\leq 1+r\,\} $$
and denote by $\mathscr{H}(\Ann_r)$ the Banach space of continuous functions $f:\Ann_r \to\C$
which are holomorphic over ${\rm int}(\Ann_r)$, endowed 
with the usual max norm  $\norm{f}_r=\max_{z\in \Ann_r} \abs{f(z)}$.
Let us fix some annulus $\Ann=\Ann_R$ 
 and introduce some measurements for non trivial holomorphic functions
in $\mathscr{H}(\Ann)$ over a compact 
sub-annulus $\Ann_\rho\subset \Ann$ where $\rho<R$.
Given a function  $f \in \mathscr{H}(\Ann)$, $f \not\equiv 0$, let  
$z_1$, $\ldots$ , $z_r$ be the zeros of $f(z)$ in ${\rm int}(\Ann)$, and $n_1$, $\ldots$, $n_r$ the corresponding multiplicities.
Set then
\begin{align*}
 N_\rho (f) &:= \sum_{\abs{\abs{z_j}-1}\leq \rho } n_j,\quad\quad  Z_\rho(f)(z)  := \prod_{\abs{\abs{z_j}-1}< \rho}  
 \left(\frac{z-z_j}{2 (R+1)}\right)^{n_k}\\
 g_\rho(f)(z)  &:= \frac{f(z)}{Z_{\rho}(f)(z)}, \qquad \beta_{\rho}(f)  := \min_{z\in \Ann_\rho}
\abs{g_{\frac{\rho+R}{2}}(f)(z)}\;.
\end{align*}
Note that since ${\rm diam}(\Ann)=2 (R+1)$, 
$\abs{Z_\rho(f)(z)}\leq 1$ for all $z\in \Ann_\rho$.
The following properties are also clear.

\begin{proposition} \label{N,beta:additivity} 
Given $f_1,f_2\in \mathscr{H}(\Ann)$, $f_1\,f_2 \not\equiv 0$
\begin{enumerate}
\item[(1)] $N_\rho(f_1\, f_2)=N_\rho(f_1) + N_\rho(f_2)$,
\item[(2)] $Z_\rho(f_1\, f_2)=Z_\rho(f_1)\,Z_\rho(f_2)$,
\item[(3)] $g_\rho(f_1\, f_2)=g_\rho(f_1)\,g_\rho(f_2)$,
\item[(4)] $\beta_\rho(f_1\, f_2)\geq \beta_\rho(f_1)\,\beta_\rho(f_2)$.
\end{enumerate}
\end{proposition}

\begin{proposition} \label{N,beta:continuity} 
Given $\rho<R$,
\begin{enumerate}
 \item[(1)]  $ \mathscr{H}(\Ann) \setminus \{0\}  \ni f \mapsto N_{\rho} (f) \in \N$  
is upper semicontinuous;
 \item[(2)]  $\mathscr{H}(\Ann) \setminus\{0\} \ni f \mapsto  \beta_{\rho} (f) \in (0, \infty)$ is lower semicontinuous.
\end{enumerate}
\end{proposition}

\begin{proof}
We begin with part (1).
Fix $f_0\in\mathscr{H}(\Ann)$, $f_0 \not\equiv0$ and take $\rho_1>\rho$ 
sufficiently close to $\rho$ so that $f_0(z)$ has no zeros in $\Ann_{\rho_1}\setminus \Ann_{\rho}$. Note that $f_0(z)$ may have zeros in $\partial \Ann_\rho$.
Then there is some neighborhood $\mathscr{U}$ of $f_0$
such that for every $f\in\mathscr{U}$, $f(z)$ has no zeros in $\partial \Ann_{\rho_1}$.
By the Argument Principle, for every $f\in\mathscr{U}$ 
$$ N_{\rho_1} (f ) =\frac{1}{2\pi i}\, \int_{\partial \Ann_{\rho_1}}
\frac{f'(z)}{f(z)}\, dz \; , $$
hence $f\mapsto  N_{\rho_1} (f)$ is continuous over $\mathscr{U}$.
Since $ N_{\rho_1} (f)\geq  N_{\rho} (f)$  and $ N_{\rho_1} (f_0) = N_{\rho} (f_0)$,
it follows that $f\mapsto  N_{\rho} (f)$ is upper semicontinuous at $f_0$.

We now turn to prove part (2).
Fix $f_0\in\mathscr{H}(\Ann)$, $f\not\equiv0$ and take $\rho_0<\rho$ 
sufficiently close to $\rho$ so that $f_0(z)$ has no zeros in ${\rm int}(\Ann_{r})\setminus {\rm int}(\Ann_{r_0})$,
where $r_0= \frac{\rho_0+R}{2}$ and $r= \frac{\rho+R}{2}$.
Note that $f_0(z)$ may have zeros in $\partial \Ann_{r}$ but not in $\partial \Ann_{r_0}$.
Hence there is some neighborhood $\mathscr{U}$ of $f_0$
such that for every $f\in\mathscr{U}$, $f(z)$ has no zeros in $\partial \Ann_{r_0}$.
Thus, denoting by $z_j=z_j(f)$ the zeros of $f(z)$
such that $r_0 <\abs{\abs{z_j}-1}< r$,
and by $n_j=n_j(f)$ their respective multiplicities, we have
$$ \abs{ g_{r_0}(f)(z) }
= \prod_{j} \abs{\frac{z-z_j}{2 (R+1)}}^{n_j} \cdot \abs{ g_{r}(f)(z) }\leq \abs{ g_{r}(f)(z) }\;,
 $$
hence
$$\eta(f):= \min_{z\in\Ann_\rho} \abs{ g_{r_0}(f)(z) }
\leq \min_{z\in\Ann_\rho} \abs{ g_{r}(f)(z) }
 = \beta_\rho(f) \;.$$
Note also that $\eta(f_0)=\beta_{\rho}(f_0)$.
Therefore, the lower semi-continuity of $\beta_\rho$ at $f_0$,
follows because $\eta:\mathscr{U}\to\R$ is continuous.
This last continuity 
relies on the fact that  functions $f\in\mathscr{U}$
are always non zero near $\partial \Ann_{r_0}$.
To be more precise, the  divisor of $f(z)$ on $\Ann_{r_0}$ is a formal linear combination, ${\rm div}_{r_0}(f)=\sum_{k=1}^r n_k\, z_k$, of the zeros of $f(z)$ in $\Ann_{r_0}$,
using their multiplicities as coefficients. We identify divisors 
with finite combinations of point mass measures, and  topologize them with the weak-$\ast$ topology. The argument principle implies that
$f\mapsto {\rm div}_{r_0} (f)$ is continuous on $\mathscr{U}$. 
If $\mu$ denotes the measure
associated with the divisor ${\rm div}_{r_0}(f)$, the polynomial 
$Z_{r_0}(f)(z)$ can be expressed as the following integral    
$$ Z_{r_0}(f)(z)=\exp\left\{ \int_\C  \log\left( \frac{z-w}{2 (R+1)} \right)\,d\mu(w) \right\}\;,$$
where `$\log$' denotes any branch of the logarithm function that contains
the zeros $z_k$ of $f(z)$ in $\Ann_{r_0}$.
This shows that the map $\mu\to Z_{r_0}(f)$ is continuous
on $\mathscr{U}$, and hence so is
$f\mapsto Z_{r_0}(f)$.
Let us now prove that  $\mathscr{U}\ni f\mapsto g_{\Ann_{r_0}}(f)\in\mathscr{H}(\Ann_{r_0})$ is continuous.
Because $Z_{r_0}(f)(z)$ does not vanish on $\partial \Ann_{r_0}$, the quotient
$g_{r_0}(f) =f/Z_{r_0}(f)$ depends continuously on $f$ w.r.t. the norm $\norm{f}_{\partial \Ann_{r_0}}$. Then Cauchy's integral formula
$$ g_{r_0}(f)(z) = \frac{1}{2\pi i}\,\int_{\partial \Ann_{r_0} } \frac{g_{{r_0}}(f)(\zeta)}{\zeta-z}\, d\zeta\;,$$
shows that $f\mapsto g_{\Ann_{r_0}}(f)$ is continuous on $\mathscr{U}$,
hence so is the function $\eta:\mathscr{U}\to \R$  defined by
$\eta(f)=\min_{z\in\Ann_\rho} \abs{ g_{r_0}(f)(z)} $.

\end{proof}

\begin{proposition} \label{lower:bound:f(z)}  Given $0<\delta <\rho$, $N > 0$, $\beta > 0$, there is some 
$\ep_0 = \ep_0 (\rho, \delta, N, \beta) > 0$
such that given two concentric  annuli $\Ann' \subset \Ann_\rho$, $\Ann'$ of width $2\delta$, and given any function $f(z)$ holomorphic over $\Ann_\rho$ which satisfies 
$N_{\rho} (f) \leq N$ and $\beta_{\rho} (f) > \beta$, there is at least one circle
$\Circ \subseteq {\rm int}(\Ann')$ such that $\abs{f (z)}\geq  \ep_0$ for every $z \in \Circ$.
\end{proposition}

\begin{proof}
Take $R>\rho$, arbitrary close to $\rho>0$ and set
 $\ep_0 = \ep_0 (\rho, \delta, N, \beta):= \beta \,\left(\frac{\delta}{2 (R+1)\,N}\right)^N = \beta\, \left(\frac{\eta_0}{2 (R+1)}\right)^N \sim \beta \,\left(\frac{\delta}{2 (\rho+1)\,N}\right)^N$ , where $\eta_0 = \frac{\delta}{N}$.  Let $z_1, \ldots , z_r$ be the zeros in $\Ann_\rho$ of some function $f	\in	\mathscr{H}(\Ann_R)$ 
with multiplicities $n_1, \ldots , n_r$. 
Then  $n_1 + \ldots + n_r=N_\rho(f)\leq N$, and,
because $N \, \eta_0= \delta <2\,\delta$, the width of $\Ann'$,  there is at least one
circle $\Circ$ concentric with $\Ann_\rho$ which does not intersect any of the
disks $D_{\eta_0} (z_k )$. Hence, for every
$z \in \Ann_\rho \setminus \cup_{i=1}^ r D_{\eta_0} (z_k )$, and in particular for every $z \in \Circ$, we have 
$\abs{f(z)}=\abs{g_{\rho}(f)(z)}\,\prod_{i=1}^r \abs{\frac{z-z_i}{2 (R+1)}}^{n_i}\geq \beta\, \left(\frac{\eta_0}{2 (R+1)}\right)^N = \ep_0$.
\end{proof}

\begin{remark} \rm{As observed in the proof, the constant $\ep_0>0$ in
proposition~\ref{lower:bound:f(z)} is explicitely given by
\begin{equation}\label{ep0-formula}
\ep_0 (\rho, \delta, N, \beta) \sim \beta \,\left(\frac{\delta}{2 (\rho+1)\,N}\right)^N
\end{equation}}
\end{remark}

\bigskip

For any $V\in C^\omega_\rho(\T,\Sym_d(\R))$ and $E \in \R$, we define the function 
$\displaystyle f_E (V) :\Ann_\rho \to\C$ by 
$$f_E(V) (z) := \det(V(z)-E \cdot I)$$
Clearly $f_E (V) \in \mathscr{H} (\Ann_\rho)$. Moreover,
 if $V$ has no constant eigenvalues,  then all functions $f_E(V)$ are non trivial,
i.e. $f_E(V)\not{\!\!\equiv} 0$, hence we have $N_{\rho}(f_E(V))<\infty$ and $\beta_{\rho}(f_E(V))>0$.

Since $\R \ni E \mapsto f_E (V) \in  \mathscr{H} (\Ann_\rho)$ is continuous  for any fixed potential $V\in\V$, from proposition~\ref{N,beta:continuity} it follows that the map 
$$\displaystyle \R \ni E \to N_{\rho}(f_E(V)) \in \N$$ 
is upper semi-continuous, while the map
$$\displaystyle \R \ni E \mapsto \beta_{\rho}(f_E(V)) \in (0, \infty)$$
is lower semi-continuous.

This shows that their maximum and respectively minimum values are attained on a compact interval of energies $E$. Since clearly for large enough values of $E$ the functions $f_E (V) (z)$ have no zeros, we can then take $E$ over the whole set of reals and define the following uniform in $E$ measurements on the potential function $V$.

\begin{definition} \label{hat:functions}
For any  $V \in \V$, the open set of potential functions in $ C^\omega_\rho(\T,\Sym_d(\R))$ with no constant eigenvalues, let   
$$\widehat{N}_\rho(V):=\max_{ E\in\R} N_{\rho}(f_E(V)) < \infty$$
$$\widehat{\beta}_\rho(V):=\min_{E\in\R} \beta_{\rho}(f_E(V)) > 0$$
\end{definition}

\begin{proposition}\label{continuity_hats} The two maps defined above satisfy the following:
\begin{enumerate}
\item[(1)] $\V \ni V\mapsto \widehat{N}_\rho(V) \in \N$
is upper semi-continuous;

\item[(2)] $\V \ni V\mapsto \widehat{\beta}_\rho(V) \in (0, \infty)$
is lower semi-continuous.
\end{enumerate}

\end{proposition}

\begin{proof}
Since the map $\R\times \mathscr{V}\to \mathscr{H}(\Ann)$, $(E,V) \mapsto f_E(V)$,  is continuous,
by proposition~\ref{N,beta:continuity} the functions
$\widetilde{N},\widetilde{\beta}:\R\times \mathscr{V}\to \R$ defined by
$\widetilde{N}(E,V):= N_{\rho}(f_E(V))$ and $\widetilde{\beta}(E,V):=  \beta_{\rho}(f_E(V))$
are respectively upper and lower semi-continuous. Fix some $\ep>0$.
Given $V_0\in\mathscr{V}$, consider the compact interval $I=[-2\norm{V}_\rho,2\norm{V}_\rho]$,
and take a neighbourhood $\mathscr{U}_0\subseteq \mathscr{V}$ of $V_0$
such that for $V\in\mathscr{U}_0$,  $E\notin I$, and $z\in\strip_\rho$,
$\abs{f_E(V)(z)}=\abs{\det (V(z)-E)}\geq 1$. This shows that for $V\in\mathscr{U}_0$
and $E\notin I$, $N_\rho(f_E(V))=0$ and $\beta_\rho(f_E(V))\geq 1$, thus proving that
$\widehat{N}_\rho(V)=\sup_{ E\in I} N_{\rho}(f_E(V))$
and
$\widehat{\beta}_\rho(V)=\sup_{ E\in I} N_{\beta}(f_E(V))$.
Because $\widetilde{N}$ and $\widetilde{\beta}$ are semi-continuous, 
for each $E\in I$ we can take an open interval $J_E\subseteq \R$,
with $E\in J_E$, and a neighbourhood $\mathscr{U}_E\subseteq \mathscr{V}$ of $V_0$
such that for every $(E',V')\in J_E\times\mathscr{U}_E$,
\begin{align*}
N_\rho(f_{E'}(V')) &\leq N_\rho(f_{E}(V_0))\leq \widehat{N}_\rho(V_0),\; \text{ and } \\
\beta_\rho(f_{E'}(V')) &\geq \beta_\rho(f_{E}(V_0)) - \ep \geq \widehat{\beta}_\rho(V_0) -\ep\;.
\end{align*}
Since $I$ is compact there are energies $E_1,\ldots, E_\ell\in I$ such that
$I\subseteq \cup_{j=1}^\ell J_{E_j}$. Hence, setting
$\mathscr{U}=\mathscr{U}_0\cap \cap_{j=1}^ \ell \mathscr{U}_j$,
we have for every $V\in\mathscr{U}$  and every $E\in I$, with $E\in J_{E_j}$,
\begin{align*}
N_\rho(f_{E}(V)) &\leq N_\rho(f_{E_j}(V_0))\leq \widehat{N}_\rho(V_0),\; \text{ and } \\
\beta_\rho(f_{E}(V)) &\geq \beta_\rho(f_{E_j}(V_0)) - \ep \geq \widehat{\beta}_\rho(V_0) -\ep\;.
\end{align*}
Thus, for every $V\in\mathscr{U}$ 
$$\widehat{N}_\rho(V)\leq \widehat{N}_\rho(V_0) 
\; \text{ and }\;
\widehat{\beta}_\rho(V) \geq \widehat{\beta}_\rho(V_0) -\ep \;, $$
which proves that $\widehat{N}_\rho$ is upper semi-continuous and
$\widehat{\beta}_\rho$ is lower semi-continuous.
\end{proof}

For holomorphic matrix-valued functions $\Dmat:\Ann_\rho\to\Mat_d(\C)$
we use the notation:
$$ N_\rho(\Dmat):= N_\rho( \det \Dmat ) \ \text{ and } \
\beta_\rho(\Dmat):= \beta_\rho( \det \Dmat ) $$

\begin{corollary}\label{U:bounds}Given $0<\delta <\rho$, $N \in \N$, $\beta > 0$, there is some
$\ep_0 = \ep_0 (\rho, \delta, N, \beta) > 0$
such that for any annulus $\Ann' \subset \Ann_\rho$
of width $\delta$, and for any function $\Dmat \in C_\rho^\omega(\T,\Mat_d)$
 with 
$N_\rho(\Dmat) \leq N$ and $\beta_\rho(\Dmat) > \beta$,  
 there is a circle
$\Circ \subseteq {\rm int}(\Ann')$ such that:
\begin{equation}\label{detlbU}
\abs{\det [ \Dmat(z) ] }\geq  \ep_0 \quad \text{ for all } z \in \Circ
\end{equation}
\end{corollary}

\begin{proof} Simply apply Proposition~\ref{lower:bound:f(z)} to the holomorphic function $f (z) := \det ( U(z) )$ on $\Ann_\rho$.
\end{proof}

\begin{corollary}\label{V:bounds}  Given $0<\delta <\rho$, $N_1, N_2 \in \N$, $\beta_1, \beta_2 > 0$, there is some
$\ep_0 = \ep_0 (\rho, \delta, N_1, N_2, \beta_1, \beta_2) > 0$
such that for any $E\in\R$, any annulus $\Ann' \subset \Ann_\rho$
of width $\delta$, and any functions $\Dmat \in C_\rho^\omega(\T,\Mat_d)$,
 $V \in C_\rho^\omega(\T,\Sym_d)$ with 
$N_\rho(\Dmat) \le N_1$, $\widehat{N}_\rho (V ) \leq N_2$ and $\beta_\rho(\Dmat) > \beta_1$, $\widehat{\beta}_\rho (V) > \beta_2$,  
 there is a circle
$\Circ \subseteq {\rm int}(\Ann')$ such that:
\begin{equation}\label{detlbUV}
 \abs{\det [ \Dmat (z) \cdot (V(z)-E \cdot I) ] }\geq  \ep_0 \quad \text{ for all } z \in \Circ
 \end{equation}
\end{corollary}

\begin{proof} Fix $E \in \R$ and let 
$$f (z) := \det [ \Dmat (z) \cdot (V(z)-E \cdot I) ] = \det [ \Dmat (z) ] \cdot \det [ V(z)-E \cdot I ]$$ 
which is clearly a holomorphic function on $\Ann_\rho$.

Using Proposition~\ref{N,beta:additivity} we have:
\begin{align*}
N_\rho (f) &= N_\rho (\det [ \Dmat (z) ] ) + N_\rho (\det [ V(z)-E \cdot I ]) \\
& \le N_\rho (U) + \widehat{N}_\rho (V ) \le N_1 + N_2 =: N < \infty 
\end{align*}
\begin{align*}
\beta_\rho (f) &\ge \beta_\rho (\det [ \Dmat (z) ] ) \cdot \beta_\rho (\det [ V(z)-E \cdot I ]) \\
&\ge \beta_\rho (U) \cdot \widehat{\beta}_\rho (V ) \ge \beta_1 \cdot \beta_2 =: \beta > 0 
\end{align*}

Now simply apply  Proposition~\ref{lower:bound:f(z)} to the function $f(z)$. Note that $\ep_0$ does not depend on $f (z)$ per se, but only on the uniform measurements $N$ and $\beta$ of $f (z)$. In particular, $\ep_0$ is independent of $E$.
\end{proof}

\smallskip

\section{The growth lemma}\label{growth-lemma}
Throughout this paper, an appropriate measure of the size of a square matrix will be its minimum expansion. For completeness we will review some of the properties of this quantity.

Let $P \in \Mat_d (\C)$ be a square matrix of dimension $d$. The minimum expansion of $P$ is defined as:
$$m (P) := \min \{ \norm{ P x } \colon \norm{ x } = 1 \} =  \text{the least singular value of } P$$

Clearly 
$$P \in \GL_d (\C) \text{ if and only if } m (P) > 0$$ 

Since the norm of $P$ is given by
$$\norm{P} = \max  \{ \norm{ P x } \colon \norm{ x } = 1 \} =  \text{the largest singular value of } P$$
we have that
$$0 \le m (P) \le \norm{P}$$
and
$$m (P) = \norm{P^{-1}}^{-1}$$

Moreover, this implies that the minimum expansion is super-multiplicative:
$$m (P_1 \cdot P_2) \ge m (P_1) \cdot m (P_2) \ \text{ for any } P_1, P_2 \in \Mat_d (\R)$$

Since $ \abs{ \det {P} } $ is the product of the singular values of $P$, we have:
$$\frac{\abs{ \det (P) }}{\norm{P}^{d-1}} \le m (P) \le \abs{\det (P) }^{1/d}$$

This inequality says that given an upper bound on the norm of the matrix $P$, having a lower bound on the minimum expansion of $P$ is the same as having a lower bound on the determinant of $P$. 

In particular, if $\abs{ \det (P) } \ge \ep$ and if $\norm{P} \le B$ then
\begin{equation}\label{minexp-det}
m (P) \ge \frac{\ep}{B^{d-1}}
\end{equation}

We will also need the following inequalities. If  $P \in \Mat_d (\C)$ such that $\norm {P} \le \delta < 1$ then
\begin{equation}\label{inv(I+P)}
\norm{ (I+P)^{-1} } \le \sum_{n \ge 0} \ \norm{ P }^n \le \frac{1}{1-\delta} 
\end{equation}
and
\begin{equation}\label{m(I+P)}
m  (I+P) =  \norm{ (I+P)^{-1} }^{-1} \ge 1 - \delta 
\end{equation}

\medskip

We will consider block matrices whose upper left corner blocks are large (i.e. their minimum expansions are large) and show that the norm of their product grows. 

All block matrices below are in $\Mat_m (\C)$ and their upper left corner blocks are in $\Mat_d (\C)$, where $1 \le d \le m$.

\begin{lemma}\label{indgrowth}
Consider the product of block matrices
$$\left[ \begin{array}{ccc}
L   & &   \Wf   \\ 
& & \\ 
\Ws  & &  O \\  \end{array} \right]  \cdot  \left[ \begin{array}{ccc}
S   & &   *   \\ 
& & \\ 
T  & &  * \\  \end{array} \right]   = \left[ \begin{array}{ccc}
L  S +  \Wf T & &  *   \\ 
& & \\ 
\Ws S + O T  & &  * \\  \end{array} \right] =: \left[ \begin{array}{ccc}
\tilde{S}   & &   *   \\ 
& & \\ 
\tilde{T}  & &  * \\  \end{array} \right] $$
where $*$ stands for blocks that do not matter here.
Assume that:
$$m (L) \ge \la \ \text{ and } \ \norm{\Wf}, \norm{\Ws}, \norm{O} \le B$$
where $\la$ and $B$ are some positive constants such that $\la > 3 B$.  

If $m (S) \ge \mu$ and $\norm{ T \, S^{-1} } < 1$ then similar estimates hold for the corresponding blocks of the product:
$$m (\tilde{S}) \ge (\la - B) \cdot \mu \  \text{ and } \ || \tilde{T} \,  \tilde{S}^{-1} || < 1$$
\end{lemma}

\begin{proof} We write 
$$L  S +  \Wf T = [ I + \Wf \ (T \ S^{-1}) \ L^{-1} ] \cdot L S$$
Since
$$\norm{  \Wf \ (T \ S^{-1})  \ L^{-1} } \le \norm{ \Wf } \cdot \norm{ T \ S^{-1} } \cdot \norm{ L^{-1} } < \frac{B}{\la} $$
according to \eqref{m(I+P)} we have:
$$m \ [ I +  \Wf \ (T \ S^{-1}) \ L^{-1} ]  \ge 1 - \frac{B}{\la}$$
We then conclude
\begin{align*}
m (L  S +  \Wf T) & \ge  \, m \, [ I + \Wf \ (T \ S^{-1}) \ L^{-1} ]  \cdot m (L) \cdot m (S) \\
& \ge \left(1 - \frac{B}{\la}\right) \, \la \, \mu = (\la - B) \, \mu
\end{align*}
Moreover
\begin{align*}
 \tilde{T} \, \tilde{S}^{-1} &= (\Ws S + O T) \cdot (L  S +  \Wf T)^{-1} \\
& =  (\Ws \ S + O T) \cdot (L S)^{-1} \cdot  [ I +  \Wf \ (T \ S^{-1}) \ L^{-1} ]^{-1} \\
& = (\Ws \ S \cdot S^{-1} L^{-1} + O \cdot (T \ S^{-1}) \cdot L^{-1}) \cdot  [ I +  \Wf \ (T \ S^{-1}) \ L^{-1} ]^{-1} 
\end{align*}
Then
$$\norm{ \tilde{T} \, \tilde{S}^{-1} }  \le \left(\frac{B}{\la} + B \cdot 1 \cdot \frac{1}{\la}\right) \cdot \frac{1}{1 - \frac{B}{\la}} = \frac{2 B}{\la - B} < 1$$
since $\la > 3 B$.

\end{proof}

\begin{lemma}\label{growth}
Consider the product of block matrices
$$M_n = \left[ \begin{array}{ccc}
S_n   & &   *   \\ 
& & \\ 
T_n  & &  * \\  \end{array} \right] := \prod_{j=n}^{1} \ \left[ \begin{array}{ccc}
L_j   & &   \Wf_j   \\ 
& & \\ 
\Ws_j  & &  O_j \\  \end{array} \right] $$
If for all $1 \le j \le n$ we have
$$m (L_j) \ge \la \ \text{ and } \ \norm{\Wf_j}, || \Ws_{j} ||, \norm{O_j} \le B$$
where $\la > 3 B$, then
\begin{equation}\label{growthmexp}
m (S_n) \ge (\la - B)^n
\end{equation}
In particular, we get:
\begin{equation}\label{growthnorm}
\norm{ M_n } \ge (\la - B)^n
\end{equation}

\end{lemma}

\begin{proof} The statement follows immediately by induction from Lemma \ref{indgrowth}. 
The only thing we need to verify is that $\norm{ \Ws_1 \ L_1^{-1} } < 1$.
Indeed: $$\norm{ L_1^{-1} } = [ m (L_1) ]^{-1} < \frac{1}{\la}$$ so
$$\norm{ \Ws_1 \ L_1^{-1} } <  B \ \frac{1}{\la} < 1$$
\end{proof}

\bigskip

Given a real (or complex) vector space $V$, consider formal $k$-products
$v_1\wedge \ldots \wedge v_k$ of vectors in $V$,
which we assume to be skew-symmetric in the sense that
for any permutation $\sigma\in {\rm S}_k$,
$$  v_{\sigma_1} \wedge \ldots \wedge v_{\sigma_k} 
= (-1)^{{\rm sgn}(\sigma)} v_1\wedge \ldots  \wedge v_k
 \;.  $$
The linear space spanned by all such formal $k$-products
is called the {\em $k$-exterior power of $V$} 
and denoted by $\wedge_k  V$.
We shall briefly recall some of the properties of this
exterior product construction,
which can be found in~\cite{Federer}.
Let $\Lambda_k^n$ be the set of all $k$-subsets 
$I=\{i_1,\ldots, i_k\}\subset \{1,\ldots, n\}$,
with $i_1<\ldots <i_k$,
and  order it lexicographically. 
Given a basis $\{e_1, \ldots, e_n\}$ of $V$,
define for each $k$-subset $I\in\Lambda^n_k$, the $k$-exterior product $e_I=e_{i_1}\wedge \ldots\wedge e_{i_k}$.
The ordered family $\{e_I\,:\,I\in\Lambda^n_k\}$ is a basis of $\wedge_k V$.
Any linear map $T:V\to V$ induces 
a linear map $\wedge_k T:\wedge_k V\to \wedge_k V$
such that $\wedge_k T(v_1\wedge \ldots \wedge v_k)=
T(v_1)\wedge \ldots \wedge T(v_k)$, for given vectors
$v_1,\ldots, v_k\in V$. This construction is functorial
in the sense that $\wedge_k(T'\circ T)=\wedge_k T'\circ  \wedge_k T$, whenever $T,T':V\to V$ are linear maps.
Given $I,J\in\Lambda^ n_k$,
let $A_{I\times J}$ be the square submatrix of $A$
with indices $(i,j)\in I\times J$.
If a linear map $T:V\to V$ is represented by the matrix
$A=(a_{ij})_{i,j}$ relative to a basis $\{e_1, \ldots, e_n\}$,
 then the $k$-exterior power
$\wedge_k T:\wedge_k V\to \wedge_k V$ is represented
by the matrix $(\det A_{I\times J})_{I,J}$  relative to  the basis $\{e_I\,:\,I\in\Lambda^n_k\}$.
Each inner product in $V$ induces an inner product
in $\wedge_k V$ such that if $\{e_1, \ldots, e_n\}$ is an orthonormal
basis in $V$ then $\{e_I\,:\,I\in\Lambda^n_k\}$ is also an 
orthonormal basis of $\wedge_k V$.
In particular, if $T:V\to V$ is an orthogonal transformation
then so is $\wedge_k T:\wedge_k V\to \wedge_k V$.
Given some eigenbasis $\{e_1, \ldots, e_n\}$ of $T$,
with associated eigenvalues $\lambda_1,\ldots, \lambda_n\in\C$,
then $\wedge_k T(e_I)=(\prod_{i\in I} \lambda_i)\, e_I$,  which shows that $\{e_I\,:\,I\in\Lambda^n_k\}$
is an  eigenbasis  of $\wedge_k T$.
Hence, if $T$ is represented by a diagonal matrix $D$ in the first basis,
then $\wedge_k T$ is represented by the diagonal matrix $\wedge_k D$ in the second basis.
The singular value decomposition   $M=R_1\,D\,R_2$, with $R_1$, $R_2$ orthogonal (or unitary) matrices, and $D$ diagonal,
induces the singular value decomposition  $\wedge_k M=(\wedge_k  R_1)\, (\wedge_k  D)\, (\wedge_k  R_2)$
at the level of $k$-exterior products.
The singular values of $\wedge_k M$ are the products of $k$ singular values of $M$.
In particular, the minimum expansion $m(M)$ is the product of the $k$ lowest singular values of $M$.

\begin{corollary}\label{c-growth}
Consider matrices $M_n\in\Mat_{m}(\C)$, $S_n\in\Mat_{d}(\C)$ satisfying the assumptions of Lemma~\ref{growth}.

Then for any $1\leq k \leq d$, we have
 \begin{equation} \label{exter:growthmexp}
m\left( \wedge_k S_n\right)\geq (\lambda -B)^{k n} \;,
\end{equation}
 and in particular,
 \begin{equation}  \label{exter:growthnorm}
\norm{\wedge_k M_n}\geq (\lambda -B)^{k n} \;.
\end{equation}
\end{corollary}

\begin{proof}
Let $0 <s_1^{(n)} \leq s_2^{(n)}\leq \ldots \leq s_r^{(n)}$
be the singular values of $S_n$.
Then 
$$m(\wedge_k S_n) = s_1^{(n)}\ldots s_k^{(n)}\geq (s_1^{(n)})^k = m(S_n)^k\geq (\lambda -B)^{k n}\;.$$
Because $S_n$ is the submatrix of $M_n$ indexed in $\{1,\ldots, r\}$
it follows that $\wedge_k S_n$ is the submatrix of $\wedge_k M_n$ indexed in $\Lambda_k^r$.
Hence 
$$\norm{\wedge_k M_n} \geq \norm{\wedge_k S_n}\geq   m(\wedge_k S_n) \geq (\lambda -B)^{k n}$$
\end{proof}

\smallskip

\section{An estimate on the mean of a subharmonic function}\label{subharmonic_section}
The following result is the main analytic tool used in this paper to establish lower bounds on Lyapunov exponents. It is based on a convexity argument for means of subharmonic functions.

\begin{proposition}\label{subharmonic}
Let $u (z)$ be a subharmonic function on a neighborhood of the annulus $\Ann_\rho = \{ z \colon 1-\rho \le \abs{z} \le 1+\rho \}$. Assume that:
\begin{align}
u(z) \le S & \ \text{ for all } z: \abs{z} = 1 + \rho \label{upperbound}\\
u (z) \ge \gamma & \ \text{ for all } z: \abs{z} = 1 + y_0 \label{y=y_0}
\end{align}
where $0 \le y_0 < \rho$.

Then
\begin{align}\label{sh-est}
\int_{\T} u (x) \, d x \ge \frac{1}{1-\alpha} \, ( \gamma - \alpha S) &
\end{align}
where
\begin{equation}\label{alphachoice}
\alpha = \frac{\log (1 + y_0) }{ \log (1 + \rho) } \sim \frac{y_0}{\rho}
\end{equation}
\end{proposition}

\begin{proof}
The proof is a simple consequence of a general result on subharmonic functions, used to derive Hardy's convexity theorem (see Theorem 1.6 and the Remark following it in \cite{Duren}). This result says that given a subharmonic function $u(z)$ on an annulus, its mean along concentric circles is $\log$ - convex. That is, if we define
$$m (r) := \int_{\abs{z} = r} \, u (z) \, \frac{d z}{2 \pi}$$
and if
\begin{equation} \label{logr}
\log r  = (1 - \alpha) \log r_1 + \alpha \log r_2 
\end{equation}
for some $ 0 < \alpha < 1$, then
\begin{equation}\label{logconv}
m (r) \le (1-\alpha) \, m (r_1) + \alpha \, m (r_2)
\end{equation}

It can be shown, using say Green's theorem, that if $u(z)$ were {\em harmonic}, then $m (r)$ would be $\log$ - affine. Then the above result for subharmonic functions would follow using the principle of harmonic majorant (see \cite{Duren} for details).

We apply \eqref{logconv} with $r = 1+y_0$, $r_1 = 1$, $r_2 = 1 + \rho$, so for \eqref{logr} to hold, $\alpha$ will be chosen as in \eqref{alphachoice}. Then the convexity property \eqref{logconv} implies:
\begin{equation} \label{logconvap}
m (1 + y_0) \le (1-\alpha) \, m (1) + \alpha \, m (1 + \rho)
\end{equation}
where
\begin{align}
m (1) & =  \int_{\abs{z} = 1}  u (z) \,  \frac{d z}{2 \pi}  = \int_{\T} \, u (x) d x \label{m1eq}\\
m (1+\rho) & =  \int_{\abs{z} = 1+\rho}  u (z) \, \frac{d z}{2 \pi}  \le S \label{mrhoeq}\\
m (1+y_0) & =  \int_{\abs{z} = 1+y_0} u (z) \, \frac{d z}{2 \pi}  \ge \gamma \label{my0eq}
\end{align}
where  \eqref{mrhoeq} and \eqref{my0eq} are due to \eqref{upperbound} and \eqref{y=y_0} respectively. 

The estimate~\eqref{sh-est} then follows from \eqref{logconvap} - \eqref{my0eq}.

\end{proof}

\begin{remark} \rm{In a previous version of this paper, we derived a similar estimate via a more complicated argument which used harmonic measures. This simpler approach and the reference \cite{Duren} were suggested to the second author by Barry Simon. This convexity argument in fact improves our previous lower bound on the mean of $u (x)$ along the torus $\T$ (because of the extra factor $\frac{1}{1-\alpha}$) and this in turn improves the lower bounds on the first $d$ Lyapunov exponents. 
}
\end{remark}

\smallskip

\section{The proof of the main statements}\label{proof}

The conclusion in Theorem \ref{main-allE} is stronger than the one in Theorem \ref{main}, since the lower bounds on the Lyapunov exponents for the family of cocycles $(T, \Ae)$ hold uniformly in $E \in \R$. The assumptions are also stronger: the potential function $V (x)$ has no constant eigenvalues vs. the potential $V (x)$ does not have $0$ as a constant eigenvalue. The size $\la_0$ of the coupling constant is also larger in Theorem~\ref{main-allE}, since it depends on the stronger uniform bounds \eqref{unifbound-allE-V} rather than on the weaker uniform bounds \eqref{unifbound}.

We present here only the proof of Theorem \ref{main-allE}, since modulo some obvious modifications, the proof of Theorem \ref{main} is similar.

\begin{proof}  For every $1 \le k \le d$ and for every $n \ge 1$, let:
\begin{equation}\label{uk(x)}
\uk (x) = \uk_n (x ; \la, E) := \frac{1}{n} \ \log \norm{ \wedge_k M_n (x ; \la,  E) } 
\end{equation}
where
\begin{equation}\label{Mn(x)}
M_n (x ; \la, E) := \prod_{j=n-1}^{0} \ \Ae (T^j \, x) 
\end{equation}
is the $n$th iteration of the cocycle $(T, \Ae (x))$, and $\wedge_k  \, M_n (x ; \la, E) $ is the $k$-exterior power of $M_n (x ;  \la, E)$ as defined in Section~\ref{growth-lemma}.

From \eqref{formula-lyap-i} and \eqref{formula-lyap} we have:
\begin{equation}\label{sum-lyap}
\LE{1} (\Ae) + \ldots + \LE{k} (\Ae) = \lim_{n \to \infty} \ \int_{\T} \uk_n (x; \la, E)  \, d x 
\end{equation}

Since $\Ae (x) \in \Arho (\T, \Mat_m (\R))$, the cocycle has a holomorphic extension $\Ae (z) \in \Mat_m (\C)$ to a neighborhood of the annulus $\Ann_\rho = \{ z \colon 1 - \rho \le \abs{z} \le 1 + \rho \}$. The transfer matrices and their exterior powers will also have holomorphic extensions to the same neighborhood of $\Ann_\rho$. Therefore, the functions $\uk (x) = \uk_n (x ; \la, E)$ defined above have {\em subharmonic} extensions
$$
\uk (z) = \uk_n (z ; \la, E) := \frac{1}{n} \ \log \norm{ \wedge_k M_n (z ; \la,  E) } 
$$
to a neighborhood of the annulus $\Ann_\rho$.

Given the upper bounds~\eqref{supnorm-all} on the norms of the blocks forming the cocycle $\Ae$, for every $z \in \Ann_\rho$ we have
$$\norm{\Ae (z)} \le \la \cdot B \, (B + \abs{E})$$
which then implies:
$$\norm{ M_n (z ; \la, E) } \le [ \la \cdot B \, (B + \abs{E}) ]^{n}$$
and
$$\norm{\wedge_k  \, M_n (z ; \la, E) } \le [ \la \cdot B \, (B + \abs{E}) ]^{k n}$$

Hence for all $1 \le k \le d$ we have:
\begin{equation}\label{ub-ukE}
\uk (z) \le k \log \, [ \la \cdot B \, (B + \abs{E}) ]
\end{equation} 

In particular, using \eqref{sum-lyap}, we have the following {\em upper} bound on the sum of the first $k-1$ Lyapunov exponents (if $k > 1$):
\begin{equation}\label{ub-sum-lyap-E}
\LE{1} (\Ae) + \ldots + \LE{k-1} (\Ae) \le (k-1)  \log \, [ \la  \cdot B \, (B + \abs{E}) ]
\end{equation}

\smallskip

For every $1 \le k \le d$, we will derive a {\em lower} bound on the sum of the first $k$ Lyapunov exponents which, when combined with \eqref{ub-sum-lyap-E} above, will lead to a lower bound on each individual Lyapunov exponent. We first explain the strategy used to obtain these lower bounds. 

As in \cite{S-S}, \cite{B}, we avoid the set $\{ x \colon m [ U (x) \, (V(x) - E \cdot I)  ] \approx 0 \}$ through complexification in the space variable $x$. Combining this with the fact that the coupling constant $\la$ is large enough, we will ensure that the assumptions in the growth lemma~\ref{growth} are met along a circle $\{ z \colon \abs{z} = 1 + y_0 \}$ which is close enough to the torus $\T$ (i.e. $y_0 \ll \rho$). This will give us a lower bound on the subharmonic function $\uk (z)$ along the circle $\{ z \colon \abs{z} = 1 + y_0 \}$. An upper bound for these functions along the (outer) circle $\{ z \colon \abs{z} = 1 + \rho \}$ follows from \eqref{ub-ukE}. 

Using then the subharmonic estimate \eqref{sh-est} in Proposition~\ref{subharmonic}, we will get a lower bound on the mean of the  subharmonic function $\uk (z)$ along the torus $\T$, which in turn, using \eqref{sum-lyap} will lead to the lower bound on the sum of the first  $k$ Lyapunov exponents. 

\medskip

We need to distinguish between the case when the parameter $E$ is large relative to the function $V$, and the case when $E$ is close to the range of the function $V$.

\smallskip

$\blob$ Assume that $E$ is large relative to the bound $B$ on the sup norm of $V$, say $\abs{E} > 2 B$. In this case, for every $1 \le k \le d$, we will obtain upper and lower bounds both of order $k \, \log (\la \, \abs{E})$.

\smallskip

From \eqref{ub-ukE} we get the following upper bound on the subharmonic function $\uk (z)$ on the outer circle:
\begin{equation}\label{upperb-largeE}
\uk (z) \le k  \log (\la  \abs{E} ) + k \log \, (\frac{3}{2} \, B) =: \Sk \ \text{ for all } z \colon \abs{z} = 1+\rho
\end{equation}

We now derive the lower bound. Let
$$L_{\la, E} (z) := \la \,\Dmat (z) \cdot (V (z) - E \cdot I)$$
be the upper left corner of  $\Ae (z)$.

For all $ z \in \Ann_\rho$ we have:
$$V (z) - E \cdot I =  ( - E \cdot I ) \cdot [ I  - E^{-1} \cdot V (z) ] $$

But
$$m (-E \cdot I) =  \abs{E} $$
and
$$\norm{E^{-1} \cdot V (z)} < (2 B)^{-1} \cdot B = \frac{1}{2}$$
so by \eqref{m(I+P)} 
$$m [ I - E^{-1} \cdot V (z)] \ge 1 - \frac{1}{2} = \frac{1}{2}$$

It follows that for all $z \in \Ann_\rho$:
\begin{equation}\label{m[V]-largeE}
m [ V(z) - E \cdot I] \ge m (- E \cdot I) \cdot m [ I -  E^{-1} \cdot V (z)]  > \frac{\abs{E} }{2}
\end{equation}

Fix $0 < \delta \ll \rho$, to be specified at the end of the proof,  and let $\Ann^\prime := \{ z \colon 1+\delta \le \abs{z} \le 1+2\delta \}$. 

Given the bounds \eqref{unifbound-allE-U} on $U (z)$, applying Corollary~\ref{U:bounds} we get: there is $\ep_0 = \ep_0 (\rho, \delta, N_1, \beta_1) > 0$ and there is a circle $\Circ = \{ z \colon \abs{z} = 1+y_0 \} \subseteq {\rm int}(\Ann')$, hence $y_0 \sim \delta$ such that:
$$
\abs{\det [ \Dmat(z) ] }\geq  \ep_0 \quad \text{ for all } z  \colon \abs{z} = 1+y_0
$$
Since also $\norm{U (z)}_\rho \le B$, from  \eqref{minexp-det} we get:
\begin{equation}\label{m[U]-largeE}
m [ \Dmat(z) ] \geq  \frac{\ep_0}{B^{d-1}} =: \ep_1 \quad \text{ for all } z  \colon \abs{z} = 1+y_0
\end{equation}

Combining \eqref{m[U]-largeE} and \eqref{m[V]-largeE} we conclude:
\begin{equation}\label{m[L]-largeE}
m [ L_{\la, E} (z) ] >  \la \frac{\abs{E}}{2} \, \ep_1  \quad \text{ for all } z  \colon \abs{z} = 1+y_0
\end{equation}

If we choose 
\begin{equation}\label{lambda-large1}
\la > \frac{3 }{\ep_1} =: \la_0
\end{equation}
then from \eqref{m[L]-largeE} we have
\begin{equation}\label{m[L]-largeE2}
m [ L_{\la, E} (z) ] > \la \frac{\abs{E}}{2} \, \ep_1 > 3 B  \quad \text{ for all } z  \colon \abs{z} = 1+y_0
\end{equation}

Since every circle centered at $0$ is invariant under the complex extension of the rotation $T  x = x + \omega$, then \eqref{m[L]-largeE2} applies to  $z_j := T^j \, z$ for all $z  \colon \abs{z} = 1+y_0$ and all $ j \ge 0$.  
Therefore, the assumptions in the growth lemma \ref{growth} apply to the matrices
$$\Ae (z_j) = \left[ \begin{array}{ccc}
L_{\la, E} (z_j)  & &   W^\flat  (z_j)  \\ 
& & \\ 
W^\sharp  (z_j)  & &  O  (z_j) \\  \end{array} \right]  
$$
Since
$$M_n (z_j ; \la, E) = \prod_{j=n-1}^{0} \, \Ae (z_j ; \la, E)$$
using Corollary~\ref{c-growth} we conclude that
\begin{equation}\label{Mny=y0}
\norm{\wedge_k  M_n (z ; \la, E)}   \ge (\la \abs{E} \, \frac{\ep_1}{2} - B )^{k n}  >  (\la \abs{E} \, \frac{\ep_1}{3} )^{k n}  
\end{equation}
provided we choose $\la > \la_0$.

For every $z  \colon \abs{z} = 1+y_0$ we then have:
\begin{align*}
\uk (z) = \uk_n (z ; \la, E) = \frac{1}{n} \, \log \norm{\wedge_k  M_n (z ;  \la, E)} 
\ge k \, \log (\la \abs{E} \, \frac{\ep_1}{3} )
\end{align*}
Hence
\begin{equation}\label{lowerb-largeE}
\uk (z) \ge k \log (\la \abs{E}) - k \log \, \frac{3}{\ep_1} =: \gak \ \text{ for all } z \colon \abs{z} = 1+y_0
\end{equation}

Given the upper bound \eqref{upperb-largeE} and the lower bound \eqref{lowerb-largeE}, we may apply  Proposition~\ref{subharmonic} and conclude from \eqref{sh-est} that
\begin{align*}
\int_{\T} \uk (x) d x \ge \frac{1}{1-\alpha} \, (\gak - \alpha \, \Sk) = \\
\frac{1}{1-\alpha} \, [ ( k \log (\la \abs{E}) - k \log \, \frac{3}{\ep_1} ) - \alpha ( k  \log (\la  \abs{E}) + k \log \, \frac{3}{2} \, B ) ] = \\
k \log (\la \abs{E}) - k \, \frac{1}{1-\alpha} \, [ \log \, \frac{3}{\ep_1} + \alpha \, \log \, \frac{3}{2} \, B  ] 
\end{align*} 

Hence if $0 < \alpha < \frac{1}{2}$, then
\begin{equation}\label{est-largeE}
\int_{\T} \uk (x) d x \ge k \log (\la \abs{E}) - k \log \, \frac{27 B}{2 \ep_1^2}
\end{equation}
provided $\la > \la_0$. 

From \eqref{est-largeE} and \eqref{sum-lyap} we get:
\begin{equation}\label{lb-sum-lyap-largeE}
\LE{1} (\Ae) + \ldots + \LE{k} (\Ae) \ge k \log (\la \abs{E}) - k \log \, \frac{27 B}{2 \ep_1^2}
\end{equation}

From \eqref{upperb-largeE} and   \eqref{sum-lyap} we have:
\begin{equation}\label{ub-sum-lyap-largeE}
\LE{1} (\Ae) + \ldots + \LE{k-1} (\Ae) \le (k-1) \,  \log (\la \abs{E}) + k \, \log \, \frac{3 B}{2} 
\end{equation}

Combining \eqref{lb-sum-lyap-largeE} and \eqref{ub-sum-lyap-largeE}, we conclude that if $\la > \la_0 = \frac{3}{\ep_1}$, then 
\begin{align}\label{conclusion-largeE}
 \LE{k} (\Ae) 
& \ge \log \, (\la \abs{E} )  - k \log \, \frac{81 B^{2}}{4 \ep_1^2}
\end{align}
which shows that when  $E$ is large relative to the function $V (x)$, the first $d$ Lyapunov exponents are all of order $\log (\la \abs{E})$.

\medskip

$\blob$ Assume that the parameter $E$ is near the range of the function $V (x)$.
Then  from now on, $E$ will be a fixed parameter in the compact set $[-2 B, 2 B]$.

From \eqref{ub-ukE} we get the following upper bound on the subharmonic function $\uk (z)$ on the outer circle:
\begin{equation}\label{upperb-sE}
\uk (z) \le  k \, \log \la + k \, \log (3 B^2)  =: \Sk \ \text{ for all } z \colon \abs{z} = 1+\rho
\end{equation}

Fix $0 < \delta \ll \rho$ to be specified at the end of the proof.

Since the function $V (x)$ satisfies the uniform bounds \eqref{unifbound-allE-V}, by Corollary~\ref{V:bounds} there is $\epob = \epob (\rho, \delta, N_1, N_2, \beta_1, \beta_2) > 0$ such that for the fixed parameter $E$, there is a circle  $\Circ = \{ z \colon \abs{z}  = 1 + y_0\}$, where $y_0 \sim \delta$ so that along this circle we have:
$$
 \abs{\det [ \Dmat (z) \cdot (V(z)-E \cdot I) ] }\geq  \epob \quad \text{ for all } z \colon \abs{z} = 1+y_0
$$

Combining this lower bound on the determinant with the upper bound $\norm{\Dmat (z) \cdot (V(z)-E \cdot I)} \le 3 B^2$ on the norm, and using 
\eqref{minexp-det} we conclude that for all $z \colon \abs{z} = 1+y_0$ we have:
\begin{equation}\label{m>ep1}
m  [  \Dmat (z) \cdot (V(z)-E \cdot I)]  \ge \frac{\epob}{(3 B^2)^{d-1}} =: \epub 
\end{equation}

Then along this circle, we have the following lower bound on the minimum expansion of the upper left corner $L_{\la, E} (z)$ of the cocycle:
\begin{equation}\label{largeblocks1}
m  [ L_{\la, E} (z)  ]  \ge \la \  \epub \quad  \text{ for all } z \colon \abs{z} = 1+y_0
\end{equation}

If we choose $\la$ such that
\begin{equation}\label{lambda-large2}
\la > \frac{3 B}{\epub} =: \overline{\la_0}
\end{equation}
then from \eqref{largeblocks1} we have
\begin{equation}\label{largeblocks}
m  [ L_{\la, E} (z)  ]  \ge \la \  \epub  > 3 B  \ \text{ for all } z \colon \abs{z} = 1+y_0
\end{equation}

As before,  \eqref{largeblocks} holds for 
$z_j := T^j \, z$, for all $z \colon \abs{z} = 1+y_0$ and for all $j \ge 0$. 
Therefore, the assumptions in the growth lemma \ref{growth} apply to the matrices
$\Ae (z_j)$ and
using Corollary~\ref{c-growth} we conclude that
\begin{equation}\label{Mny=y0}
\norm{\wedge_k  M_n (z ; \la, E)}   \ge ( \la \  \epub - B)^{k n} \ge  ( \la \, \frac{2}{3} \,  \epub )^{k n}
\end{equation}

For every $z  \colon \abs{z} = 1+y_0$ we then have:
\begin{align*}
\uk (z) = \uk_n (z ; \la, E) = \frac{1}{n} \, \log \norm{\wedge_k  M_n (z ;  \la, E)} 
\ge k \log \,  \la \, \frac{2}{3} \,  \epub
\end{align*}
Hence
\begin{equation}\label{lowerb}
\uk (z) \ge  k \log \la - k \log \, \frac{3}{2 \epub} =: \gak \ \text{ for all } z \colon \abs{z} = 1+y_0
\end{equation}

We can now apply Proposition \ref{subharmonic} to the functions $\uk (x) = \uk_n (x)$ and conclude that:
\begin{align*}
\int_{\T} \uk (x) d x \ge \frac{1}{1-\alpha} \, (\gak - \alpha \, \Sk) = \\
\frac{1}{1-\alpha} \, [ ( k \log \la - k \log \, \frac{3}{2 \epub} ) - \alpha ( k  \log \la + k \log \, 3  B^2 ) ] = \\
k \log \la - k \, \frac{1}{1-\alpha} \, [  \log \, \frac{3}{2 \epub}  + \alpha \log 3 B^2 ]
\end{align*} 

Hence if $0 < \alpha < \frac{1}{2}$, then
\begin{equation}\label{est-sE}
\int_{\T} \uk (x) d x \ge k \, \log \la - k \, \log \, \frac{27 B^2}{4 \epub^2}
\end{equation}
provided $\la > \overline{\la_0}$.

From \eqref{est-sE} and \eqref{sum-lyap} we get:
\begin{equation}\label{lb-sum-lyap-sE}
\LE{1} (\Ae) + \ldots + \LE{k} (\Ae) \ge  k \, \log \la - k \, \log \, \frac{27 B^2}{4 \epub^2}
\end{equation}

From \eqref{upperb-sE} and   \eqref{sum-lyap} we have:
\begin{equation}\label{ub-sum-lyap-sE}
\LE{1} (\Ae) + \ldots + \LE{k-1} (\Ae) \le (k-1) \,  \log  \la  + k \, \log \, 3 B^2
\end{equation}

Combining \eqref{lb-sum-lyap-sE} and \eqref{ub-sum-lyap-sE}, we conclude that if $\la > \overline{\la_0} =\frac{3 B}{\epub}$ then 
\begin{align}\label{conclusion-smallE}
 \LE{k} (\Ae) 
& \ge \log \la  - k \log \, \frac{81 B^{4}}{4 \epub^2}
\end{align}

\medskip

We now indicate how $\delta$, the width of the annulus $\Ann^\prime$ (where we find the circles along which we have uniform hyperbolicity)  is chosen. 

Since $$\alpha = \frac{\log (1 + y_0)}{\log (1+\rho)}$$ and since $y_0 < 2 \delta$, to ensure that $0 < \alpha < \frac{1}{2}$, it is enough to choose 
$$0 < \delta < \frac{\sqrt{1+\rho}-1}{2}$$
 \end{proof}

\begin{remark}\label{size-la}\rm{
In the case of parameters $E$ in a bounded interval $[-2 B, 2 B]$, to obtain the lower bound \eqref{conclusion-smallE} on the first $d$ Lyapunov exponents of the cocycle, the coupling constant $\la$ needs to be chosen such that:
$$\la > \overline{\la_0} =\frac{3 B}{\epub}$$
where
\begin{equation}\label{minexp-ep1}
m  [  \Dmat (z) \cdot (V(z)-E \cdot I)]  \ge \epub
\end{equation}
holds along some circle which is close enough to the torus $\T$.

This shows that the threshold $\overline{\la_0}$ for the size of the coupling constant, and the lower bounds \eqref{conclusion-smallE} on the first $d$ Lyapunov exponents do {\em not}, a-priori, depend on the dimension $d$ of the upper left corner block of the cocycle. They depend only on the lower bound of its minimum expansion along a circle which is close enough to the unit circle, and on the sup norms of the blocks of the cocycle. 

The calculations above show that $\la_0$ can be estimated {\em explicitly} from the measurements on the matrix blocks forming the cocycle. Those estimates may involve the dimension $d$ of the upper left corner block, if an estimate on the minimum expansion is obtained using  \eqref{minexp-det}, i.e. via a lower bound on the determinant and an upper bound on the norm.  

However, if one has an independent procedure for estimating the minimum expansion of the upper left corner block of the cocycle, then the dimension $d$ would not enter the estimates on the coupling constant and on the lower bounds of the Lyapunov exponents.  

A similar observation applies to the case of large parameters $E$, the only difference being that the relevant quantity is the minimum expansion of the factor $\Dmat (z)$.
 
}
\end{remark}

\smallskip

\section{Applications and extensions of the main statements}\label{consequences}

Standard examples of linear cocycles are Schr\"{o}dinger coycles associated with lattice Schr\"{o}dinger operators.

\smallskip  

Given a potential function $ v \in \Arho (\T, \R)$, a frequency $\om \in \R \setminus \Q$ and a coupling constant $\la > 0$, consider the quasi-periodic \textit{integer} lattice Schr\"{o}dinger operator $H_{\la, x}$ acting on square summable sequences of real numbers $l^2 (\Z, \R)$ by
\begin{equation}\label{s-op} 
[H_{\la, x} \, \psi]_n := - \psi_{n+1} - \psi_{n-1} + \la \, v (x + n \om) \, \psi_n
\end{equation}

The associated Schr\"{o}dinger equation 
\begin{equation}\label{s-eq} 
[H_{\la, x} \, \psi]_n = - \psi_{n+1} - \psi_{n-1} + \la \, v  (x + n \om) \, \psi_n = E \, \psi_n
\end{equation}
can be written as
\begin{equation}\label{s-eq2}
\left[ \begin{array}{cc}
\psi_{n+1} \\
\psi_n 
 \end{array} \right]  = A_{\la, E} (x + n \om) \cdot
 \left[ \begin{array}{cc}
\psi_{n} \\
\psi_{n-1} 
 \end{array} \right]  
\end{equation}
where
\begin{equation}\label{s-cocycle}
A_{\la, E} (x) = \left[ \begin{array}{ccc}
\la \, v ( x )  - E     &   - 1 \\ 
1 &  0  \\  \end{array} \right]  \in \SL_{2} (\R)
\end{equation}
is called the Schr\"{o}dinger (family of) cocycle(s) associated to the  equation \eqref{s-eq}.

\medskip

$\blob$ Assuming that the potential function $v (x)$ is non-constant, and that the coupling constant $\la$ is large enough,  positivity of the Lyapunov exponent for the family of cocycles \eqref{s-cocycle} is given by Sorets-Spencer's theorem (see \cite{S-S}).

\medskip

More generally, consider the quasi-periodic \textit{band} lattice Schr\"{o}dinger operator  $H_{\la, x}$  acting on 
$l^2 (\Z \times \{1, \ldots d \}, \R) \cong l^2 (\Z, \R^d)$, $d\ge1$ by
\begin{equation}\label{band-s-op} 
[H_{\la, x} \, \vpsi]_n := - \vpsi_{n+1} - \vpsi_{n-1} + V_\la (x + n \om) \, \vpsi_n
\end{equation}
where $V_\la \in \Arho (\T, \Sym_d (\R))$ and  $\vpsi_n$ is regarded as a vector in $\R^d$. 

\smallskip

$\blob$ If $ V_\la (x)$ is diagonal and $V_\la (x) = \la \, \text{ diag } [v_1 (x), \ldots,  v_d (x)]$, then the corresponding $2 d$-dimensional  Schr\"{o}dinger cocycle is a direct sum of $2$-dimensional Schr\"{o}dinger cocycles. Therefore, if each of the diagonal entries $v_j (x)$ is a non-constant analytic function,  positivity of the $d$ largest Lyapunov exponents of $A_{\la, E} (x)$ is a direct consequence of Sorets-Spencer's one-dimensional result. Moreover, other one-dimensional results and methods extend to this model (although not in such a straightforward manner, but involving much more effort): J. Bourgain and S. Jitomirskaya proved (see \cite{BJ}) Anderson localization for this diagonal model with large coupling constant $\la$.  

\medskip

$\blob$ Positivity of the first $d$ Lyapunov exponents for the case of a {\em constant} perturbation of the diagonal  was established by  I. Ya. Goldsheid and E. Sorets (see \cite{SG}).

More precisely, the result holds for the operator  \eqref{band-s-op}  with $V_\la (x) = \la \, \text{ diag } [v_1 (x), \ldots,  v_d (x)] - R$, where $R \in \Sym_d (\R)$ is a constant matrix, each diagonal entry $v_j (x)$ is a non-constant analytic function, and $\la$ is large enough.

The corresponding Schr\"{o}dinger operator in this case can of course be written in the form:
$$
[H_{\la, x} \, \vpsi]_n := - (\vpsi_{n+1} + \vpsi_{n-1} + R \, \vpsi_n) + \la \, D (x + n \om) \, \vpsi_n
$$
where
$$D (x) =   \text{diag } [v_1 (x), \ldots,  v_d (x)]  = \left[ \begin{array}{ccc}
v_1  ( x )  & \ldots   &   0  \\ 
\vdots & \ddots & \vdots\\ 
0 & \ldots &  v_d (x) \\  \end{array} \right]  
$$

\bigskip

Given the generality of the cocycle $A_{\la, E} (x)$ we have defined in \eqref{cocycle-allE}, Theorem~\ref{main-allE} will apply to a much more general version of the above cocycles. This application  includes cocycles associated to quasi-periodic Jacobi operators or to quasi-periodic, finite range hopping lattice or band lattice  Schr\"{o}dinger operators. 

Recall from section~\ref{main_section} that given ``weight''  functions $W \in  \Arho (\T, \Mat_d (\R))$, $R \in  \Arho (\T, \Sym_d (\R))$ and a potential function $D \in  \Arho (\T, \Sym_d (\R))$, if we denote
$$
W_n (x) := W (x+n \om), R_n (x) := R (x+n \om), D_n (x) := D (x+n \om)
$$
then we can define the self-adjoint operator $H_{\la, x} $ on $l^2 (\Z, \R^d)$ by
\begin{equation}\label{J-op-ap}
[ H_{\la, x} \, \vpsi ]_n := - (W_{n+1} (x) \, \vpsi_{n+1} + W^{T}_{n} (x) \, \vpsi_{n-1} + R_{n} (x) \, \vpsi_{n}) + \la \, D_n (x) \, \vpsi_n 
\end{equation}

Consider the associated Schr\"{o}dinger equation
\begin{equation}\label{J-eq-ap}
- (W_{n+1} (x) \, \vpsi_{n+1} + W^{T}_{n} (x) \, \vpsi_{n-1} + R_{n} (x) \, \vpsi_{n}) + \la \, D_n (x) \, \vpsi_n 
= E \, \vpsi_n
\end{equation}

We are now ready to prove our main application, namely Theorem~\ref{main-ap}.

\begin{proof}
The Schr\"{o}dinger equation \eqref{J-eq-ap} can be written in the form:
$$
W_{n+1} (x) \, \vpsi_{n+1} =  [ \la \, D_n (x) - R_{n} (x) - E \cdot I ] \, \vpsi_n - W^{T}_{n} (x) \, \vpsi_{n-1}  
$$

To simplify notations,  replace $E$ by $E / \la$ and denote:
$$V (x)  := D (x) - \la^{-1} \, R (x) \ \text{ and } \ V_n (x) := V (x + n \om)$$

Then the above equation becomes:
\begin{equation}\label{J-eq3-ap}
\vpsi_{n+1} =  \la \, W_{n+1}^{-1}  (x) \cdot  [ V_n (x) - E \cdot I ] \, \vpsi_n - W_{n+1}^{-1}  (x) \cdot  W^{T}_{n} (x) \, \vpsi_{n-1}  
\end{equation}

Writing it as a first order finite differences vectorial equation, we get:
\begin{equation}\label{J-eq2-ap}
\left[ \begin{array}{cc}
\vpsi_{n+1} \\
\\
\vpsi_n 
 \end{array} \right]  = 
A_{\la, E} (x + n \om) \cdot
 \left[ \begin{array}{cc}
\vpsi_{n} \\
\\
\vpsi_{n-1} 
 \end{array} \right]  
\end{equation}
where
\begin{align*}
 A_{\la, E} (x + n \om)  
 = \left[ \begin{array}{ccc}
\la \, W_{n+1}^{-1} (x) \, (V_n (x) - E \cdot I)      &   - W_{n+1}^{-1} (x) \cdot W_{n}^{T} (x)  \\ 
\\
I  &  \zero  \\  \end{array} \right]  
\end{align*}

Therefore, the linear cocycle associated with \eqref{J-eq-ap} is given by
\begin{equation} \label{J-cocycle}
 A_{\la, E} (x) := 
 \left[ \begin{array}{ccc}
\la \, W^ {-1} (x + \om) \, (V (x) - E \cdot I)      &   - W^{-1} (x + \om) \cdot W^{T} (x)  \\ 
\\
I  &  \zero  \\  \end{array} \right] 
\end{equation}

The matrix valued function $W (x)$ is not necessarily invertible for every value of $x \in \T$. However,  due to \eqref{TC-W} we have
$$g (x) := \det [ W (x) ]  \not \equiv 0$$
Moreover, since $W (x)  \in  \Arho (\T, \Mat_d (\R))$, $g (x)$ also has a holomorphic extension, hence $g (x)  \in  \Arho (\T, \R)$.

Therefore, $g (z)$ has finitely many zeros in the annulus $ \Ann_\rho$, so there are finitely many values of $x \in \T$ for which $W^{-1} (x)$  is not defined.

Hence the cocycle  $A_{\la, E} (x)$ in \eqref{J-cocycle} and its iterates are defined for all but  a countable (hence negligible) set of phases $x \in \T$. In particular the Lyapunov exponents of this cocycle are well defined. 

We replace the cocycle $A_{\la, E} (x)$ by one with no singularities, for which Theorem~\ref{main-allE} applies, and transfer all singularities to a one dimensional cocycle. 

By Cramer's rule, for all $x \in \T$, 
$$W (x) \cdot \adjW (x) = \det [ W (x) ] \cdot I$$
or, for all but finitely many $x \in \T$,
$$W^{-1} (x) = \frac{1}{\det [ W (x) ]} \cdot \adjW (x)$$
where $\adjW (x)$ is the adjugate matrix of $W (x)$ (i.e. the transpose of the matrix whose entries are the minors of $W (x)$). 

Then if we multiply the cocycle  $A_{\la, E} (x)$ by $g (x+\om) =  \det [ W (x+\om) ]$, we obtain the cocycle
 \begin{equation}\label{J-adjcocycle}
 \tilde{A}_{\la, E} (x) := 
 \left[ \begin{array}{ccc}
\la \, \adjW (x + \om) \, (V (x) - E \cdot I)      &   - \adjW (x + \om) \cdot W^{T} (x)  \\ 
\\
g (x+\om) \cdot I  &  \zero  \\  \end{array} \right]  
\end{equation}

If we set
\begin{align*}
U (x) & :=  \adjW (x + \om) \\
\Wf (x) & := - \adjW (x + \om) \cdot W^{T} (x)  \\ 
\Ws (x) & := g (x+\om) \cdot I \\
O (x) & := \zero
\end{align*}
then the hypotheses  \eqref{unifbound-allE-U} -  \eqref{TC-allE} of Theorem~\ref{main-allE} apply to the cocycle $ \tilde{A}_{\la, E} (x)$.
Indeed:
$$\det [ U (x) ] = \det [ \adjW (x + \om) ] = (g (x+\om))^{d-1} \not \equiv 0$$
so \eqref{TC-U} holds, and in particular  $N_\rho (U) =: N_1 < \infty$ and $\beta_\rho (U) > \beta_1 > 0$, which establish  \eqref{unifbound-allE-U}.

From the assumption \eqref{TC-D}, the matrix $D (x)$ has no constant eigenvalues, i.e. $D(x) \in \V$. Since $\V$ is open, there is $\ep = \ep (D) > 0$ such that if 
$$\norm{\la^{-1} \, R }_\rho = \frac{\norm{R}_\rho}{\la} < \ep$$
then $$V (x) = D (x) - \la^{-1} \, R (x) \in \V$$
provided $\la$ is large enough depending on $D$ and $R$. 
This means that $V (x)$ has no constant eigenvalues either, hence \eqref{TC-allE} holds.

Moreover, since $D (x)  \in \V$, we have $\widehat{N}_\rho (D) < \infty$, $\widehat{\beta}_\rho (D) > 0$. By proposition~\ref{continuity_hats}, the uniform estimates  $\widehat{N}_\rho$ and  $\widehat{\beta}_\rho$ are upper semi-continuous and lower semi-continuous respectively. Then for $\la$ large enough depending on $D$ and $R$, we have:
\begin{align*}
\widehat{N}_\rho (V) & = \widehat{N}_\rho (D - \la^{-1} \, R) \le \widehat{N}_\rho (D) + 1 =: N_2 < \infty \\
\widehat{\beta}_\rho (V) & = \widehat{\beta}_\rho (D -  \la^{-1} \, R) \ge \frac{\widehat{\beta}_\rho (D)}{2} =: \beta_2  > 0
\end{align*}  
which shows that \eqref{unifbound-allE-V} holds as well.

Finally, assuming $\la > 1$, we clearly have:
\begin{align*}
\norm{U}_\rho & = \norm{\adjW}_\rho \le \norm{W}_\rho^{d-1}\\
\norm{V}_\rho & \le \norm{D}_\rho + \la^{-1} \, \norm{R}_\rho \le \norm{D}_\rho + \norm{R}_\rho \\
 \norm{\Wf}_\rho & \le  \norm{W}_\rho^{d-1} \cdot  \norm{W}_\rho = \norm{W}_\rho^{d} \\
 \norm{\Ws}_\rho & \le \norm{W}_\rho^d
\end{align*}
Putting
$$B := \max \{ \norm{W}_\rho^{d-1}, \norm{W}_\rho^{d},  \norm{D}_\rho + \norm{R}_\rho \} < \infty$$
we have that \eqref{supnorm-all}  holds as well.

Theorem~\ref{main-allE} then applies to the cocycle $\tilde{A}_{\la, E} (x)$ and we have: there are constants 
 $\la_0 = \la_0 (W, R, D)$ and $c = c (W, R, D)$ such that if $\la > \la_0$, then:
 \begin{equation}\label{proof-adj}
 L_k (\tilde{A}_{\la, E} ) \ge \log \la - c \quad \text{for all } E\in \R, \ 1 \le k \le d 
 \end{equation}

\smallskip

From the definition of the cocycle \eqref{J-adjcocycle}, it is clear that
$$
\tilde{A}_{\la, E} (x) = g (x+\om) \cdot A_{\la, E} (x)
$$
Then if 
$$\tilde{M}_n (x) = \tilde{M}_n (x ; \la, E) := \prod_{j=n-1}^{0} \, \tilde{A}_{\la, E} (x+j \om) $$
are the transfer matrices of the cocycle $\tilde{A}_{\la, E} (x)$, we have
\begin{equation*}
\frac{1}{n} \, \log \norm{\tilde{M}_n (x)} = \frac{1}{n} \, \sum_{j=1}^{n} \log \abs{g (x+ j \om)} + 
\frac{1}{n} \, \log \norm{M_n (x)} 
\end{equation*}
and
\begin{equation*}
\frac{1}{n} \, \log \norm{\wedge_k \, \tilde{M}_n (x)} = k \cdot \frac{1}{n} \, \sum_{j=1}^{n} \log \abs{g (x+ j \om)} + 
\frac{1}{n} \, \log \norm{\wedge_k \, M_n (x)} 
\end{equation*}
for all $1 \le k \le d$.

Then for every $1 \le k \le d$ we have:
\begin{equation}\label{dec-cocycle-poof}
L_k (\tilde{A}_{\la, E}) = L (g) + L_k (A_{\la, E}) 
\end{equation} 
where $L (g)$ is the Lyapunov exponent of the one dimensional cocycle $ g (x)$.

Since $g(x)$ has a holomorphic extension to the annulus $\Ann_\rho$, and since $g (x) \not \equiv 0$, it is easy to see by factoring out its zeros that $\log \abs{g (x)} \in L^1 (\T)$. In particular, by Birkhoff's ergodic theorem, 
$$\lim_{n \to \infty} \, \frac{1}{n} \, \sum_{j=1}^{n} \log \abs{g (x+ j \om)}  = \int_{\T} \, \log \abs{g (x)} d x \quad \text{ for a.e. } x \in \T $$
hence
\begin{equation}\label{proof-ap-g}
L (g) =  \int_{\T} \, \log \abs{g (x)} d x = \int_{\T} \, \log \abs{\det [W (x) ] } d x  
\end{equation}
which is a real number that depends on $W (x)$.

Then the  bound \eqref{main-ap-c} on the first $d$ Lyapunov exponents in the conclusion of the theorem follows from \eqref{dec-cocycle-poof}, \eqref{proof-adj} and \eqref{proof-ap-g}.

\smallskip

To conclude that the other $d$ Lyapunov exponents are the additive inverses of the first $d$ exponents, we will show that  the cocycle $A_{\la, E} (x)$ in \eqref{J-cocycle} is conjugated to a {\em symplectic} cocycle, for which this property holds automatically.

Indeed, as before, replace $E$ by $E / \la$ and denote $V (x) := D (x) - \la^{-1} \, R (x)$. We can write the Schr\"{o}dinger equation \eqref{J-eq-ap} in the form:
$$
\left[ \begin{array}{cc}
W_{n+1} (x) \vpsi_{n+1} \\
\\
\vpsi_n 
 \end{array} \right]  =   A^W_{\la, E} (x + n \om)  \cdot
 \left[ \begin{array}{cc}
W_n (x)  \vpsi_{n} \\
\\
\vpsi_{n-1} 
 \end{array}\right]  
$$
where
\begin{equation}\label{sadel-cocycle}
A^W_{\la, E} (x) := \left[ \begin{array}{cc}
\la \, ( V (x) -  E\cdot I ) \, W^{-1} (x)   &    - W^T (x) \\ 
&  \\ 
W^{-1} (x)   &   \zero \\  \end{array} \right] 
  \in \symp_{d} (\R)
\end{equation}

The cocycle $A_{\la, E} (x)$ which corresponds to solving the Schr\"{o}dinger equation \eqref{J-eq-ap} for the vector  $\left[ \begin{array}{cc}
  \vpsi_{n+1} 
\\
\vpsi_{n} 
 \end{array} \right]$ 
is conjugated to the ``weighted'' cocycle $A^W_{\la, E} (x)$ which corresponds to solving the same equation for the ``weighted" vector $\left[ \begin{array}{cc}
W_{n+1} (x) \, \vpsi_{n+1} \\
\vpsi_{n} 
 \end{array} \right]$.
 
 Indeed, a simple calculation shows that
 \begin{equation}\label{conjugation}
 A_{\la, E} (x) = [ C (x + \om) ]^{-1} \cdot  A^W_{\la, E} (x) \cdot C (x)
 \end{equation}
 where
 $$C (x) :=   \left[ \begin{array}{ccc}
W (x)  &    \zero \\
\\
 \zero  &  I  \\  \end{array} \right] $$
 
Since $\det [ C (x) ] = \det [ W (x) ] = g (x)$, and since $\log \abs{g (x)} \in L^1 (\T)$, it is easy to verify that 
$\log \norm{C (x)}, \ \log \norm{[ C (x) ]^{-1}} \in L^1 (\T)$. Then from \eqref{conjugation} we conclude that the cocycle $ A_{\la, E} (x) $ has the same Lyapunov exponents as the symplectic cocycle $ A^W_{\la, E} (x)$.

\end{proof}

\begin{remark}\rm{ If in Theorem~\ref{main-ap} we assume that the potential $D (x) =  \text{diag } [v_1 (x), \ldots,  v_d (x)]$ is a diagonal matrix, then the assumption \eqref{TC-D} on $D (x)$ having non-constant eigenvalues simply means that the diagonal entries $v_j (x)$ are non-constant functions.

If,  moreover, we let  $R (x) \equiv R$ and $W (x) \equiv I$, then we obtain the aforementioned result of I. Ya. Goldsheid and E. Sorets (see \cite{SG}).}
\end{remark}

\medskip


Since our main statements in section~\ref{main_section} hold for real-valued cocycles of \textit{any} dimension, they can be extended to hold for complex-valued cocycles as well. Thus we will prove complex versions of theorems~\ref{main}, \ \ref{main-allE}, as well as of its main consequence,  theorem~\ref{main-ap}.

\medskip

To each complex number $a+i b\in\C$
we associate  the matrix
$$ M_{a+i b}:= \left[ \begin{array}{cc}
a & -b\\
b & a  
\end{array}\right]\in\Mat_2(\R)\;.$$
The set $\mathscr{C}=\{\, M_z\in \Mat_2(\R)\, :\, z\in\C\,\}$  is
an algebra isomorphic to the complex field $\C$.
Given any  matrix $A\in\Mat_d(\C)$, we call its {\em realification}
the real matrix
$\widetilde{A}\in\Mat_{2d}(\R)$ obtained by replacing each complex entry $a_{ij}$
by the real matrix block  $M_{a_{ij}}\in\Mat_2(\R)$.
For any square (real or complex) matrix $A$, we denote by ${\rm spec}(A)$ the spectrum of $A$, and by ${\rm sing}(A)$ its singular spectrum, i.e., the set of
singular values of $A$. A complex matrix $A\in\Mat_d(\C)$ is called {\em hermitian}
if  $A^\ast=A$. We shall denote by $\Herm_d(\C)$ the real vector space
of hermitian matrices in $\Mat_d(\C)$.

\begin{proposition}\label{matrix:realification}
 Given $A\in\Mat_d(\C)$, the realification
 $\widetilde{A}$ of $A$ satisfies the following relations:
\begin{enumerate}
\item[(a)] $\displaystyle  \det(\widetilde{A}) =  \abs{ \rm{det}_{\C}(A)}^2$;
\item[(b)] ${\rm spec}(\widetilde{A})= {\rm spec}(A)\cup {\rm spec}(\overline{A})$;
\item[(c)] ${\rm sing}(\widetilde{A})= {\rm sing}(A)$,
 but each singular value of $\widetilde{A}$ has twice the multiplicity than it has as a singular value of $A$;
\item[(d)] $\displaystyle  \|\widetilde{A}\| =  \norm{A} $\, and \, 
 $\displaystyle m( \widetilde{A} ) =  m(A) $;
\item[(e)] $\widetilde{A}\in\Sym_{2d}(\R)$ \; $\Leftrightarrow$\; $A\in\Herm_d(\C)$.
\end{enumerate}
\end{proposition}

\begin{proof}
Given $A+iB\in\Mat_d(\C)$, with $A,B\in\Mat_d(\R)$, its realification  
 is conjugate (by row and column permutation) to the block matrix
$\left[ \begin{array}{cc}
A & -B\\
B & A  
\end{array}\right]$.\,
Hence, using the determinant rule for $2\times 2$ block matrices 
\begin{align*}
\det(\widetilde{A+i B}) &= \det \left[ \begin{array}{cc}
A & -B\\
B & A  
\end{array}\right] = \det (A^2-B\,(-B))\\
&= \det (A^2+B^2)= {\det}_\C [(A+i B)(A-i B)]\\
&= {\det}_\C (A+i B)\,\overline{{\det}_\C (A+i B)}= \abs{{\det}_\C (A+i B)}^2\;.
\end{align*}
To prove part (b) just remark that
$$ \det(\widetilde{A}-\lambda\, I)= \abs{{\det}_\C(A-\lambda\,I)}^2 =
{\det}_\C(A-\lambda\,I)\, {\det}_\C(\overline{A}-\overline{\lambda}\,I)\;. $$
For the next part we observe that the realification process, $A\mapsto \widetilde{A}$,
is an algebra homomorphism such that $\widetilde{A^\ast}=\widetilde{A}^T$.
Then part (b) implies (c) because the singular values of $A$ are the square roots of
eigenvalues of $A^\ast\,A$, the singular values of $\widetilde{A}$ are the square roots of
eigenvalues of $\widetilde{A}^T\,\widetilde{A}$, and we have $\widetilde{A^\ast\,A}=\widetilde{A}^T\,\widetilde{A}$.
Part (d) is a consequence of (c) because $\norm{A}$ and $m(A)$ are, respectively,
the largest and smallest singular values of a matrix $A$.
Finally, to prove (e) take a complex matrix $A+iB$, with $A,B\in\Mat_d(\R)$.
Identifying $\widetilde{A+i B}$ with $\left[ \begin{array}{cc}
A & -B\\
B & A  
\end{array}\right]$, this matrix is symmetric if and only if $A^T=A$ and $B^T=-B$,
which is equivalent to $A+i B$ being hermitian.
\end{proof}

We define the {\em realification} of 
a complex potential $V:\T\to\Herm_d(\C)$ and that of  
a complex cocycle $A:\T\to \Mat_m(\C)$ to be 
$\widetilde{V}:\T\to\Sym_{2d}(\R)$ and respectively $\widetilde{A}:\T\to \Mat_{2m}(\R)$, where
these functions assign to each $x\in\T$ the realification $\widetilde{V}(x)$ of $V(x)$ and respectively
the realification $\widetilde{A}(x)$ of $A(x)$.

\begin{proposition} \label{LE:realification}
Any integrable cocycle $A:\T\to\Mat_m(\C)$ has the same
Lyapunov exponents as its  realification $\widetilde{A}$. 
More precisely, 
$$L^{(i)}(A) = L^{(2i-1)}(\widetilde{A}) = L^{(2i)}(\widetilde{A})\; \text{ for }\; 1\leq i\leq m\;. $$
\end{proposition}

\begin{proof}
The conclusion follows from the characterization of the Lyapunov exponents in terms of
singular vaules, mentioned in the introduction, and item (c) of proposition~\ref{matrix:realification}. 
\end{proof}

\begin{proposition} Given $V:\T\to\Herm_d(\C)$ and $A:\T\to\Mat_m(\C)$,
\begin{enumerate}
\item[(a)] \; $\widetilde{V}\in C_\rho^\omega(\T,\Sym_{2d}(\R))$\;
$\Leftrightarrow$\;  $V\in C_\rho^\omega(\T, \Herm_d(\C))$;
\item[(b)] \; $\widetilde{A}\in C_\rho^\omega(\T,\Mat_{2m}(\R))$\;
$\Leftrightarrow$\;  $A\in C_\rho^\omega(\T, \Mat_m(\C))$.
\end{enumerate}
\end{proposition}

\begin{proof}
The proof is straighforward.
\end{proof}

\medskip

With the obvious interpretations of the basic assumptions (\ref{unifbound})-(\ref{TC-allE}),
the main theorems of section~\ref{main_section} extend from real matrix valued functions to complex  matrix valued functions. More precisely   the following holds.

\begin{theorem} \label{main-Complex}
The statements of theorems~\ref{main} and ~\ref{main-allE} hold for cocycles
$A\in C_\rho^\omega(\T, \Mat_m(\C))$ of the form~\eqref{(T,A)}. In  theorem ~\ref{main-allE}
we assume $V:\T\to\Herm_d(\C)$. 
\end{theorem}

\begin{proof}
Given a hermitian potential $V:\T\to\Herm_d(\C)$, remark that  $N_\rho(\widetilde{V})=N_\rho(V)$,
$\beta_\rho(\widetilde{V})=\beta_\rho(V)$, $\|\widetilde{V}\|_\rho=\norm{V}_\rho$,
$\|\widetilde{\Lambda}\|_\rho=\norm{\Lambda}_\rho$, etc.
Also, for every $x\in\T$,  $\det \widetilde{V}(x)=\abs{\det_\C V(x)}^2$,
and for every $z\in\strip_\rho$,  $m(\widetilde{\Lambda}(z))=m(\Lambda(z))$.
Hence the assumptions (\ref{unifbound})-(\ref{TC-allE}) of theorems~\ref{main} and ~\ref{main-allE} on the
complex cocycle $A$ imply the corresponding assumptions for its realification $\widetilde{A}$.
Applying these theorems we derive the wanted conclusion on the real cocycle $\widetilde{A}$,
which by proposition~\ref{LE:realification} implies the corresponding conclusion for the complex cocycle $A$.
\end{proof}

\begin{remark}\rm{
The same complex extension, with obvious interpretations, holds for our main application, theorem~\ref{main-ap}. We of course have to assume that  $W \in  C_\rho^\omega(\T, \Mat_d(\C))$ and that $R, \, D \in  C_\rho^\omega(\T, \Herm_d(\C))$.}
\end{remark}


\section*{Acknowledgments}

The authors were partially supported by Funda\c c\~ao para a Ci\^encia e a Tecnologia through the Program POCI 2010 and the Project ``Randomness in Deterministic Dynamical Systems and Applications'' (PTDC-MAT-105448-2008).

 The second author would like to thank his hosts at Universidade de Lisboa, Portugal for their support and hospitality. He is also grateful to Christian Sadel for a useful conversation, and in particular for his description of an earlier version of the model \eqref{J-op}. 

The current version of this paper benefited greatly from the feedback the second author has received during his talks at Caltech and UC Irvine. In particular, he would like to acknowledge: Barry Simon for suggesting the use of Hardy's convexity theorem in proving Theorem~\ref{subharmonic}; Christoph Marx for suggesting that the weight function $W (x)$ in Theorem~\ref{main-ap} may not need to be invertible everywhere; Rupert Frank for clarifying what would be the more physically relevant application of our main result.

\smallskip


\nocite{*}
\bibliographystyle{amsplain} 
\bibliography{references}

\def\cprime{$'$} \def\cprime{$'$}
\providecommand{\bysame}{\leavevmode\hbox to3em{\hrulefill}\thinspace}
\providecommand{\MR}{\relax\ifhmode\unskip\space\fi MR }
\providecommand{\MRhref}[2]{%
  \href{http://www.ams.org/mathscinet-getitem?mr=#1}{#2}
}
\providecommand{\href}[2]{#2}
\begin{thebibliography}{10}

\bibitem{Ahlfors}
Lars~V. Ahlfors, \emph{Complex analysis. {A}n introduction to the theory of
  analytic functions of one complex variable}, McGraw-Hill Book Company, Inc.,
  New York-Toronto-London, 1953.

\bibitem{B}
J.~Bourgain, \emph{Green's function estimates for lattice {S}chr\"odinger
  operators and applications}, Annals of Mathematics Studies, vol. 158,
  Princeton University Press, Princeton, NJ, 2005. 

\bibitem{BG}
J.~Bourgain and M.~Goldstein, \emph{On nonperturbative localization with
  quasi-periodic potential}, Ann. of Math. (2) \textbf{152} (2000), no.~3,
  835--879.

\bibitem{BJ}
J.~Bourgain and S.~Jitomirskaya, \emph{Anderson localization for the band
  model}, Geometric aspects of functional analysis, Lecture Notes in Math.,
  vol. 1745, Springer, Berlin, 2000, pp.~67--79.

\bibitem{BrunsVetter}
Winfried Bruns and Udo Vetter, \emph{Determinantal rings}, Lecture Notes in
  Mathematics, vol. 1327, Springer-Verlag, Berlin, 1988. 
  
\bibitem{Chan}
Jackson Chan, \emph{Method of variations of potential of quasi-periodic
  {S}chr\"odinger equations}, Geom. Funct. Anal. \textbf{17} (2008), no.~5,
  1416--1478. 

\bibitem{Duren}
Peter~L. Duren, \emph{Theory of {$H^{p}$} spaces}, Pure and Applied
  Mathematics, Vol. 38, Academic Press, New York, 1970. 

\bibitem{Federer}
Herbert Federer, \emph{Geometric measure theory}, Die Grundlehren der
  mathematischen Wissenschaften, Band 153, Springer-Verlag New York Inc., New
  York, 1969. 

\bibitem{SG}
I.~Ya. Gol{\cprime}dshe{\u\i}d and E.~Sorets, \emph{Lyapunov exponents of the
  {S}chr\"odinger equation with quasi-periodic potential on a strip}, Comm.
  Math. Phys. \textbf{145} (1992), no.~3, 507--513. 

\bibitem{H}
Michael-R. Herman, \emph{Une m\'ethode pour minorer les exposants de
  {L}yapounov et quelques exemples montrant le caract\`ere local d'un
  th\'eor\`eme d'{A}rnol\cprime d et de {M}oser sur le tore de dimension
  {$2$}}, Comment. Math. Helv. \textbf{58} (1983), no.~3, 453--502.

\bibitem{Hirsch}
Morris~W. Hirsch, \emph{Differential topology}, Springer-Verlag, New York,
  1976, Graduate Texts in Mathematics, No. 33. 

\bibitem{HSY}
Brian~R. Hunt, Tim Sauer, and James~A. Yorke, \emph{Prevalence: a
  translation-invariant ``almost every'' on infinite-dimensional spaces}, Bull.
  Amer. Math. Soc. (N.S.) \textbf{27} (1992), no.~2, 217--238. 
  
\bibitem{Simon-Kotani}
S.~Kotani and B.~Simon, \emph{Stochastic {S}chr\"odinger operators and {J}acobi
  matrices on the strip}, Comm. Math. Phys. \textbf{119} (1988), no.~3,
  403--429.

\bibitem{Levin}
B.~Ya. Levin, \emph{Lectures on entire functions}, Translations of Mathematical
  Monographs, vol. 150, American Mathematical Society, Providence, RI, 1996, In
  collaboration with and with a preface by Yu.\ Lyubarskii, M. Sodin and V.
  Tkachenko, Translated from the Russian manuscript by Tkachenko.
  
\bibitem{Schlag}
Wilhelm Schlag, \emph{Regularity and convergence rates for the {L}yapunov
  exponents of linear co-cycles}, preprint (2012), 1--21.

\bibitem{S-S}
Eugene Sorets and Thomas Spencer, \emph{Positive {L}yapunov exponents for
  {S}chr\"odinger operators with quasi-periodic potentials}, Comm. Math. Phys.
  \textbf{142} (1991), no.~3, 543--566.

\end{thebibliography}

\end{document}